\documentclass[11pt]{article}

\usepackage{tikz}
\usepackage{empheq}

\usepackage[shortlabels]{enumitem}
\setlist[enumerate]{leftmargin=15mm,nosep}

\usepackage[title]{appendix}
\usepackage{fancybox}
\usepackage[doc, wmcom,jvcom]{optional}
\usepackage{xcolor}
\usepackage[colorlinks=true,
linkcolor=refkey,
urlcolor=lblue,
citecolor=red]{hyperref}
\usepackage{float}
\usepackage{soul}
\usepackage{graphicx}
\definecolor{labelkey}{rgb}{0,0.08,0.45}
\definecolor{refkey}{rgb}{0,0.6,0.0}
\definecolor{Brown}{rgb}{0.45,0.0,0.05}
\definecolor{lime}{rgb}{0.00,0.8,0.0}
\definecolor{lblue}{rgb}{0.5,0.5,0.99}
\definecolor{OliveGreen}{rgb}{0,0.6,0}
\usepackage{mathpazo}

\colorlet{hlcyan}{cyan!30}

\usepackage{stmaryrd}
\usepackage{amssymb}
\makeatletter
\def\namedlabel#1#2{\begingroup
	\def\@currentlabel{#2}%
	\label{#1}\endgroup
}
\makeatother

\newcommand{\seppthree}{\setlength{\itemsep}{-3pt}}

\newcommand{\ssnonex}{super strongly nonexpansive}

\providecommand{\siff}{\Leftrightarrow}

\newcommand{\nnn}{\ensuremath{{n\in{\mathbb N}}}}

\newcommand{\menge}[2]{\big\{{#1}~\big |~{#2}\big\}}

\newcommand{\fenv}[1]%
{\ensuremath{\,\overrightarrow{\operatorname{env}}_{#1}}}
\newcommand{\benv}[1]%
{\ensuremath{\,\overleftarrow{\operatorname{env}}_{#1}}}

\newcommand{\scal}[2]{\left\langle{#1},{#2}  \right\rangle}

\newcommand{\RR}{\ensuremath{\mathbb R}}

\newcommand{\NN}{\ensuremath{\mathbb N}}

\newcommand{\dom}{\ensuremath{\operatorname{dom}}}

\newcommand{\prox}{\ensuremath{\operatorname{Prox}}}

\newcommand{\ran}{\ensuremath{{\operatorname{ran}}\,}}
\newcommand{\zer}{\ensuremath{\operatorname{zer}}}

\newcommand{\Id}{\ensuremath{\operatorname{Id}}}
\newcommand{\bDelta}{{\begin{proof} \Delta}}
	\newcommand{\bv}{{\begin{proof} v}}

		\newcommand{\bg}{{\begin{proof} g}}

			{\begin{list}{}{%
						\settowidth{\labelwidth}{\textrm{#1~}}%
						\setlength{\leftmargin}{\labelwidth+\labelsep}}}%requires macro calc.sty
				{\end{list}}
			\usepackage{amsthm}
			\usepackage[capitalize,nameinlink]{cleveref}
			\crefname{equation}{}{equations}
			\crefname{chapter}{Appendix}{chapters}
			\crefname{item}{}{items}
			\crefname{enumi}{}{}
			\crefname{appsec}{Appendix}{Appendices}

			\newtheorem{theorem}{Theorem}[section]
			\newtheorem{lemma}[theorem]{Lemma}
			
			\newtheorem{corollary}[theorem]{Corollary}
			
			\newtheorem{proposition}[theorem]{Proposition}
			
			\newtheorem{definition}[theorem]{Definition}
			
			\newtheorem{example}[theorem]{Example}
			\newtheorem{ex}[theorem]{Example}
			\newtheorem{fact}[theorem]{Fact}
			\newtheorem{remark}[theorem]{Remark}

			\def\endproof{\ensuremath{\hfill \quad \blacksquare}}

			\providecommand{\abs}[1]{\lvert#1\rvert}
			\providecommand{\norm}[1]{\lVert#1\rVert}
			
			\providecommand{\normsq}[1]{\lVert#1\rVert^2}

			\providecommand{\LA}{\Leftarrow}
			\providecommand{\RA}{\Rightarrow}

			\providecommand{\grad}{\nabla}

			\providecommand{\RR}{\mathbb{R}}

			\providecommand{\opint}[1]{\left]#1\right[}
			\providecommand{\ocint}[1]{\left]#1\right]}

			\providecommand{\ran}{\operatorname{ran}}

			\providecommand{\dom}{\operatorname{dom}}

			\newcommand{\fix}{\ensuremath{\operatorname{Fix}}}

			\providecommand{\gra}{\operatorname{gra}}
			\providecommand{\Id}{\operatorname{{ Id}}}

			\providecommand{\kk}{{\begin{proof} K}}

				\providecommand{\fady}{\varnothing}

				\providecommand{\rras}{\rightrightarrows}

				\providecommand{\NN}{\mathbb{N}}

				\providecommand{\fix}{\operatorname{Fix}}
				\providecommand{\ran}{\operatorname{ran}}
				
				\providecommand{\Id}{\operatorname{Id}}

				\providecommand{\zer}{\operatorname{zer}}

				\providecommand{\fady}{\varnothing}

				\providecommand{\RR}{\mathbb{R}}
				\providecommand{\NN}{\mathbb{N}}

				\newcommand{\vsne}{super strongly nonexpansive}

				\newenvironment{myproof}[1][\proofname]{
					{\emph {#1} }%
				}{\endproof}
				
				\newcommand{\ds}{\displaystyle}
				\DeclareMathOperator{\domai}{dom} 
				\DeclareMathOperator{\graph}{gr}
				\definecolor{myblue}{rgb}{0.9,0.9,0.98}
				\newcommand*\mybluebox[1]{%
					\colorbox{myblue}{\hspace{1em}#1\hspace{1em}}}

				\allowdisplaybreaks % or locally if problems {\allowdisplaybreaks
				\usepackage[most]{tcolorbox}

				\newtcbox{\mymath}[1][]{%
					nobeforeafter, math upper, tcbox raise base,
					enhanced, colframe=blue!20!black,
					colback=brown!10, boxrule=0.7pt,
					#1}
				
\begin{document}
					%-------------------------------------------------------------------------

					%

					\author{
						Leon Liu\thanks{Department of Combinatorics and Optimization, 
							University of Waterloo,
							Waterloo, Ontario N2L~3G1, Canada.
							. E-mail:
							\texttt{l352liu@uwaterloo.ca}.}, ~
						Walaa M. Moursi\thanks{Department of Combinatorics and Optimization, 
							University of Waterloo,
							Waterloo, Ontario N2L~3G1, Canada.
							%and
							%Mansoura University, Faculty of Science,
							%Mathematics Department,
							%Mansoura 35516, Egypt.
							E-mail: \texttt{walaa.moursi@uwaterloo.ca}.} ~and~
						Jon Vanderwerff\thanks{
							Department of Mathematics, La Sierra University, Riverside, CA 92515, USA.
							E-mail: \texttt{jvanderw@lasierra.edu}.}
					}

					\title{\textsf{
							Strongly nonexpansive mappings revisited:
							\\
							uniform monotonicity and operator splitting
						}
					}

					\date{May 18, 2022}
					
					\maketitle	
					\begin{abstract}
						The correspondence between the class of nonexpansive mappings 
						and the class of maximally monotone operators 
						via the reflected resolvents of 
						the latter has played an instrumental role in the convergence analysis of
						the splitting methods. Indeed, the performance of 
						some of these methods, e.g., 
						Douglas--Rachford and Peaceman--Rachford methods 
						hinges on iterating the so-called splitting operator associated with the 
						individual operators.
						This splitting operator is a function of the composition of the reflected resolvents
						of the underlying operators.
						In this paper, we provide a comprehensive study of the class of uniformly 
						monotone operators and their corresponding reflected resolvents.
						We show that the latter is closely related to the class of the 
						strongly nonexpansive operators introduced by Bruck and Reich.
						Connections to duality via   inverse operators are systematically studied.
						We provide applications to Douglas--Rachford 
						and Peaceman--Rachford methods. 
						Examples that illustrate and tighten our results are presented.
					\end{abstract}
					{ 
						%\small
						\noindent
						{\bfseries 2010 Mathematics Subject Classification:}
						{49M27, %Decomposition methods
							%65K05, %Mathematical programming methods
							65K10, %Optimization and variational techniques
							90C25; %Convex programming
							Secondary 
							47H14, %Perturbations of nonlinear operators
							49M29. %Methods involving duality
							%49N15. %Duality theory
						}

						
						\section{Introduction}
						
						Throughout, we assume that 
						\begin{empheq}[box=\mybluebox]{equation}
							\text{$X$ is
								a  real Hilbert space, 
							}
						\end{empheq}
						with inner product 
						$\scal{\cdot}{\cdot}\colon X\times X\to\RR$ 
						and induced norm $\|\cdot\|$. 
						Let $A\colon X\rras X$ be a  set-valued operator.	The \emph{graph} of $A$
						is $\gra A=\menge{(x,x^*)\in X\times X}{x^*\in Ax}$.
						Recall that 
						$A$ is \emph{monotone} if $\{(x,x^*),(y,y^*)\}\subseteq \gra A$
						implies that $\scal{x-y}{x^*-y^*}\ge 0$.
						A monotone operator  $A$ is \emph{maximally monotone} 
						if $\gra A$ does not admit a proper extension (in terms of set inclusion) 
						to a graph of  
						a monotone operator.
						The
						\emph{resolvent} of $A$
						is
						$J_A=(\Id+A)^{-1} $ and the
						\emph{reflected resolvent} of $A$
						is
						$R_A=2J_A-\Id $,
						where $\Id\colon X\to X\colon x\mapsto  x$.
						
						The theory of monotone operators has been of significant interest 
						in optimization:
						indeed a typical  problem in convex optimization seeks finding 
						a minimizer of  
						the sum $f+g$, where both $f$ and $g$ are proper lower semicontinuous 
						convex functions on $X$. 
						Thanks to Rockafellar's fundamental work  (see \cite[Theorem~A]{Rock1970}) 
						the (possibly set-valued)
						\emph{subdifferential} operators $\partial f$
						and $\partial g$  of $f$ and $g$ respectively 
						are maximally monotone.
						Assuming appropriate constraint qualifications
						the problem of minimizing $f+g$ amounts to solving the 
						monotone inclusion problem: 
						\begin{equation}
							\tag{P}
							\label{P}
							\text{Find $x\in X$
								such that $x\in \zer(A+B) = \menge{x\in X}{0\in Ax+Bx}$.}
						\end{equation}	
						For a comprehensive discussion on \cref{P}
						and its connection to optimization problems 
						we refer the reader to 
						\cite{BC2017}, 
						\cite{Borwein50}, 
						\cite{Brezis}, 	
						\cite{BurIus},
						\cite{Rock98},
						\cite{Simons1},
						\cite{Simons2},
						\cite{Zeidler2a},
						\cite{Zeidler2b} and the references therein.
						Splitting algorithms are potential candidates  to solve \cref{P}.
						Many of these algorithms employ the resolvent and/or the reflected  
						resolvent of the underlaying operators $A$ and $B$.
						The monotonicity of an operator $A$ 
						is reflected in the firm nonexpansiveness 
						of its resolvents or, equivalently; the nonexpansiveness 
						of its reflected resolvent. 
						When a monotone operator $A$ posses supplementary
						properties, e.g., strong monotonicity, Lipschitz  continuity
						or cocoercivity its reflected resolvent enjoys refined notions 
						of nonexpansiveness, see, e.g., 
						\cite{BMW2021},
						\cite{BMW12},
						\cite{Gis2017},
						and 
						\cite{MVan2019}.
						However, none of these works studies the notion of  
						\emph{uniform monotonicity} and what the corresponding
						property in the reflected resolvent (if any) could be.
						
						\emph{The goal of this paper is to provide a systematic
							study of the class of uniformly 
							monotone operators and their corresponding reflected resolvents.
							We show that the latter is closely related to the class of the 
							strongly nonexpansive operators introduced by Bruck and Reich
							\cite{BR77}.
							Connection to duality via the inverse operators is systematically studied.
							When the underlying operators are subdifferentials of proper lower 
							semicontinuous convex functions better duality results hold.
							We provide applications to Douglas--Rachford, 
							Peaceman--Rachford and forward-backward algorithms                   . 
							Examples that illustrate and tighten our results are presented.
						}

						\subsection*{Organization and notation}
						The organization of this paper is as follows: 
						\cref{sec:2} contains a collection of auxiliary results and facts.
						Our main results appear in \cref{sec:3}--\cref{sec:7}. 
						In \cref{sec:3} we provide key results concerning 
						the correspondence between the class of uniformly
						monotone operators and
						the new class of super strongly nonexpansive mappings.
						In \cref{sec:4} 
						we prove the surjectivity of uniformly monotone operators.
						In \cref{sec:5} and  \cref{sec:6}  
						we demonstrate  the power of self-dual properties 
						and the connection to the class of contractions for large distances. 
						\cref{sec:7} is dedicated to the study of the 
						compositions of the classes of nonexpansive mapping studied
						in the paper.
						Finally, \cref{sec:8} presents applications 
						of our results  to refine and strengthen
						known results in operator splitting methods.

						The notation we adopt is standard and follows largely, e.g., 
						\cite{BC2017} and \cite{Rock1970}.

						\section{Facts and auxiliary results}
						\label{sec:2}
						
						We start by recalling 
						the following instrumental fact by Minty.
						
						\begin{fact}[{\bf Minty's Theorem}]{\rm\cite{Minty}
								(see also \cite[Theorem~21.1]{BC2017})}
							\label{thm:minty}
							Let $A\colon X\rras X$ be monotone. Then
							\begin{equation}
								\label{eq:Minty}
								\gra A=\menge{(J_A x, (\Id-J_A)x)}{x\in \ran (\Id+A)}.
							\end{equation}
							Moreover,
							\begin{equation}
								\label{eq:Minty:2}
								\text{$A$ is maximally monotone $\siff$ $\ran (\Id+A)=X$.}
							\end{equation}
						\end{fact}	
						
						Let $A\colon X\rras X$ be monotone
						and let $(x,u)\in X\times X$. In view of \cref{thm:minty},
						it is easy to check that
						\begin{equation}
							\label{lem:gr:RA:i}
							%		(\forall (x,u)\in X\times X)\quad
							(x,u)\in \gra J_A\siff (u, x-u) \in \gra A,
						\end{equation}
						and that
						\begin{equation}
							\label{eq:gr:A:RA}
							%			(\forall (x,u)\in X\times X)\quad
							(x,u)\in \gra R_A\siff
							\bigl(\tfrac{1}{2}(x+u),
							\tfrac{1}{2}(x-u)
							\bigr) \in \gra A.
						\end{equation}
						
						It is straightforward to verify that 
						(see, e.g., \cite[Proposition~23.38]{BC2017})
						\begin{equation}
							\label{eq:fix:JA:RA}
							\zer A=\fix J_A=\fix R_A,
						\end{equation}	
						and that 
						\begin{equation}
							\label{eq:RA:-RA}
							J_{A^{-1}}=\Id -J_A \text{ and consequently } R_{A^{-1}}=-R_A.
						\end{equation}	
						\begin{example}
							Suppose that $f\colon X\to \left]-\infty,+\infty\right]$ 
							is convex lower semicontinuous and proper.
							Then $\partial f$ is maximally monotone.
							%							and $\partial f^*=(\partial f)^{-1}$.
							The resolvent $J_A=J_{\partial f}=\prox f$
							and the reflected resolvent is $R_f= R_A=2\prox f-\Id$ and hence,
							by \cref{eq:RA:-RA},  
							$R_{f^*}=-R_A=\Id-2\prox_f$.
						\end{example}		
						\begin{proof}
							See \cite{Rock1970} and  \cite[Example~23.3]{BC2017}. 	
						\end{proof}

						Let $\phi\colon \RR_{+}\to\left[0,+\infty\right]$ 
						be an increasing function that vanishes only at $0$
						(such a function is called a modulus).
						Recall that  $f\colon X\to\left]-\infty,+\infty\right] $ is \emph{uniformly convex} with modulus $\phi$
						if $(\forall x\in \dom f)$ $(\forall y\in \dom f)$
						$(\forall \alpha\in \left]0,1\right[)$
						\begin{equation}
							f(\alpha x+(1-\alpha )y)+\alpha(1-\alpha)\phi(\norm{x-y})\le \alpha f(x)+(1-\alpha)f(y).
							\end{equation}
						Recall also that
						$A\colon X\rras X$ is \emph{uniformly monotone} with modulus 
						$\phi$ if 
%						$\phi$ is increasing, vanishes only at $0$, and 
						$\{(x,x^*),(y,y^*)\}\subseteq  \gra A$ implies that 
						\begin{equation}
							\label{eq:def:um}
							\scal{x-y}{x^*-y^*}\ge \phi(\norm{x-y}).
						\end{equation}		
					
						The notion of uniform monotonicity  is naturally motivated 
						by properties of subdifferentials of uniformly convex functions. 
						A comprehensive overview of 
						uniformly convex functions is found in \cite{Za02}. 
						Some other results regarding uniformly 
						convex functions can be found in \cite{BV10,BV12,Za83}.
						\begin{fact}
							Suppose that $f\colon X\to \left]-\infty,+\infty\right]$ 
							is uniformly convex with a modulus 
							$\phi$.
							Then $\partial f$ is uniformly monotone with a modulus $2\phi$. 
						\end{fact}	
						\begin{proof}
							See \cite[Theorem~3.5.10]{Za02} and also \cite[Example~22.4(iii)]{BC2017}.
						\end{proof}

						We now turn to definitions of certain classes of  mappings
						related to the notions of nonexpansiveness and Lipschitz continuity.	
						\begin{definition}
							\label{d:buttout}
							Let
							$T\colon X\to X$, let $(x,y)\in X\times X$
							and let $\alpha \in\left]0,1\right[$.
							\begin{enumerate}
								\item
								\label{d:nonex}
								$T$ is \emph{nonexpansive}
								if 
								%				$(\forall (x,y)\in X\times X)$
								$\norm{Tx-Ty}\le \norm{x-y}$.
								\item
								\label{d:snonex}
								$T $ is \emph{strongly nonexpansive}
								if $T$ is nonexpansive and we have the implication
								\begin{align}
									\label{eq:def:sn}
									\left.\begin{array}{r@{\mskip\thickmuskip}l}
										(x_n-y_n)_\nnn \text{ is bounded} \\
										\norm{x_n-y_n}-\norm{Tx_n-Ty_n}\to 0
									\end{array} \right\}
									\quad \implies \quad
									(x_n-y_n)-(Tx_n-Ty_n)\to 0.
								\end{align}

								\item
								\label{d:av}
								$T $ is \emph{$\alpha$-averaged}
								if $\alpha\in\left[0,1\right[$ and there exists a nonexpansive
								operator $N\colon X\to X$
								such that $T=(1-\alpha)\Id+\alpha N$;
								equivalently, 
								%				$(\forall (x,y)\in X\times X)$
								we have (see \cite[Proposition~4.35]{BC2017})
								\begin{equation}
									(1-\alpha)\norm{(\Id-T)x-(\Id-T)y}^2
									\le \alpha(\norm{x-y}^2-\norm{Tx-Ty}^2).
								\end{equation}
								\item
								\label{d:fne}
								$T $ is \emph{firmly nonexpansive}
								if $T$ is $\tfrac{1}{2}$-averaged;
								equivalently,
								%				if $(\forall (x,y)\in X\times X)$
								\begin{equation}
									\normsq{Tx-Ty}+\normsq{(\Id-T)x-(\Id-T)y}\le \normsq{x-y}.
								\end{equation}	
								\item
								$T$ is {\em Lipschitz for large distances} (see \cite[Proposition~1.11]{BL2000}) if for each $\epsilon > 0$, 
								there exists $K_{\epsilon} > 0$ so that
								$\|Tx - Ty\| \le K_\epsilon \|x - y\|$ whenever 
								$\|x - y\| \ge \epsilon$. 
							\end{enumerate}
						\end{definition}
						
						The following  well-known fact 
						summarizes the correspondences between the class of maximally monotone
						operators and the classes of firmly nonexpansive and  nonexpansive
						mappings.
						
						\begin{fact}
							\label{fact:corres}
							Let  $T\colon X\to X$,
							and
							set $A=T^{-1}-\Id$. Then $T=J_A$
							 and $2T-\Id=R_A$.
							Moreover,
							\begin{equation}
								\label{eq:corresp}
								\text{$A$ is maximally monotone $\siff$ $T$ is firmly nonexpansive  
									$\siff$ $2T-\Id$ is nonexpansive.}
							\end{equation}	
						\end{fact}
						\begin{proof}
							See, e.g.,
							\cite[Theorem~2]{EckBer} or \cite[Corollary~23.11]{BC2017}.
						\end{proof}	
						
						We conclude this section with the following fact 
						concerning the asymptotic behaviour 
						of iterates of strongly nonexpansive mappings. 
						
						\begin{fact}
							\label{fact:T:sne:converges}
							Let $T\colon X\to X$ be strongly nonexpansive.
							Suppose that $\fix T\neq \fady$.
							Let $x_0\in X$. Then there exists $\overline{x}\in \fix T$
							such that $(T^nx_0)_\nnn$ converges weakly to $\overline{x}$.
						\end{fact}	
						\begin{proof}
							See \cite[Corollary~1.1]{BR77}.
						\end{proof}

						\section{Strongly nonexpansive and \ssnonex\  mappings}
						\label{sec:3}
						The main goal of this section is to show that the
						notion of uniform monotonicity of an operator corresponds to
					the notion of \emph{super strong} nonexpansiveness 
						(see \cref{def:ssexp} below)
						of its reflected resolvent.
						Throughout we assume that 
						\begin{empheq}[box=\mybluebox]{equation}
							\text{$A\colon X\rras X$ and $B\colon X\rras X$ are
								maximally monotone.
							}
						\end{empheq}
						
						We start by defining a new subclass of nonexpansive mappings.
						\begin{definition}
							\label{def:ssexp}
							Let $T\colon X\to X$. We say that $T$ is 
							\emph{\ssnonex}\ if $T$ is nonexpansive and we have the implication
							\begin{equation}
								\norm{x_n-y_n}^2-	
								\norm{Tx_n-Ty_n}^2\to 0\RA (x_n-y_n)-(Tx_n-T y_n) \to 0. 
							\end{equation}	
						\end{definition}	
						
						\begin{proposition}
							\label{lem:une:sne}
							Let $T\colon X\to X$. Suppose  $T$ is 
							\ssnonex. Then $T$ is strongly nonexpansive.
						\end{proposition}	
						\begin{proof}
							Let $(x_n)_\nnn$ and $(y_n)_\nnn$ be sequences in $X$ such that 
							$(x_n-y_n)_\nnn$ is bounded and suppose that $\norm{x_n-y_n}-\norm{Tx_n-Ty_n}\to 0$.
							We claim that  	
							\begin{equation}
								\norm{x_n-y_n}^2-\norm{Tx_n-Ty_n}^2\to 0.
							\end{equation}	
							Indeed, the nonexpansiveness of $T$ implies that $(Tx_n-Ty_n)_\nnn$ is bounded.
							Now
							$\norm{x_n-y_n}^2-\norm{Tx_n-Ty_n}^2
							=(\norm{x_n-y_n}-\norm{Tx_n-Ty_n})(\norm{x_n-y_n}+\norm{Tx_n-Ty_n})
							\to 0$.
							Since $T $ is \ssnonex\ we conclude that 
							$(x_n-y_n)-(Tx_n-T y_n) \to 0$. Hence, $T$ is strongly nonexpansive as claimed.
						\end{proof}

						The converse of \cref{lem:une:sne} is not true in general as we illustrate
						in \cref{ex:umoce}  below. 	
						Nonetheless, when $X=\RR$ the converse of \cref{lem:une:sne} holds as we next illustrate
						in \cref{prop-sne-r}. We will use the following simple observation.
						Let $(a,b)\in X\times X$. Then
						\begin{equation}
							\label{e:observ}
							(\norm{a}-\norm{b})^2\le  \abs{\norm{a}^2-\norm{b}^2}.
						\end{equation}	
						\begin{proposition}[{\bf the case of the real line}] 
							\label{prop-sne-r} 
							Suppose $T\colon\RR \to \RR$ is strongly nonexpansive, 
							then $T$ is  \ssnonex.
						\end{proposition}
						
						\begin{proof}
							Suppose for eventual contradiction that
							$T$ is strongly nonexpansive, but $T$ is not \ssnonex. 
							Then we have sequences $(x_n)_\nnn$ and $(y_n)_\nnn$ in $\RR$ so that
							\begin{equation} \label{prop-sne-r-a}
								|x_n - y_n|^2 - |Tx_n - Ty_n|^2 \to 0
							\end{equation}
							but 
							\begin{equation} \label{prop-sne-r-b}
								(x_n - y_n) - (Tx_n - Ty_n) \not\to 0.
							\end{equation}
							It follows from the nonexpansiveness of
							$T$ and  \cref{e:observ}
							applied with $(a,b,X)$ replaced by $(x_n-y_n, Tx_n-Ty_n,\RR)$ 
							that $(\forall \nnn )$
							$0\le (|x_n - y_n| - |Tx_n - Ty_n| )^2\le |x_n - y_n|^2 - |Tx_n - Ty_n|^2 $.
							Therefore \cref{prop-sne-r-a}  implies that 
							\begin{equation} 
								\label{prop-sne-r-c}
								|x_n - y_n| - |Tx_n - Ty_n| \to 0,
							\end{equation}
							and because $T$ is strongly nonexpansive, 
							\cref{prop-sne-r-b} and 
							\cref{prop-sne-r-c}  imply $(x_n - y_n)$ is not bounded.
							Thus, without lost of generality,  and swapping $x_n$ and $y_n$ as necessary, 
							we may and do assume $(\forall \nnn )$ $y_n - x_n > 1$. 
							
							If $(Ty_n - Tx_n)_\nnn$ is eventually positive, 
							that is, the same sign as $(y_n - x_n)$ for all large $n$, 
							then \cref{prop-sne-r-c} would imply
							that $(y_n - x_n) - (Ty_n - Tx_n) \to 0$ in contradiction with \cref{prop-sne-r-b}. 
							Therefore, after passing to a subsequence and relabelling if necessary, we may and do assume that
							\begin{equation}
								\label{e:nov23:iv}
								(\forall \nnn)\;\;Ty_n - Tx_n < 0.
							\end{equation}
							Now consider $z_n = x_n + 1$, so $x_n < z_n < y_n$. 
							Observe that 
							\begin{equation}
								\label{e:nov23:ii}
								|x_n - y_n| = |x_n - z_n| + |z_n - y_n|=1+|z_n - y_n|.
							\end{equation}
							Hence, because $T$ is nonexpansive we learn,
							in view of the triangle inequality, \cref{e:nov23:ii}
							and  \cref{prop-sne-r-c}, that
							\begin{subequations}
								\begin{align}
									0&\le |x_n - z_n| -| Tx_n - Tz_n| + |z_n - y_n| -|Tz_n - Ty_n|
									\\
									&=  |x_n - y_n|-| Tx_n - Tz_n|-|Tz_n - Ty_n|\le |x_n - y_n|- |Tx_n - Ty_n| \to 0.
								\end{align}
								\label{e:nov23:i}
							\end{subequations}
							It follows from the nonexpansiveness of $T$ 
							and \cref{e:nov23:i} that $	|Tz_n - Tx_n| - |z_n - x_n| \to 0$.
							Consequently, because  $T$ is strongly nonexpansive, we have 
							$(Tz_n - Tx_n) - 1 =(Tz_n - Tx_n) - (z_n - x_n)\to 0$. 
							After passing to a subsequence and relabelling if necessary, 
							we may and do assume that
							\begin{equation}
								\label{e:nov23:iii}
								(\forall \nnn) \; \;Tz_n \ge Tx_n.
							\end{equation}
							Because $T$ is nonexpansive, in view
							of \cref{e:nov23:ii}
							we have $(\forall \nnn)$ $|Ty_n - Tz_n| \le |y_n - z_n| = |y_n - x_n| - 1$.
							This and  \cref{e:nov23:iii} imply
							$
							(\forall \nnn) \; \;Ty_n \ge Tz_n - (|y_n - x_n| - 1) \ge Tx_n - |y_n - x_n|  + 1.
							$
							Therefore, by \cref{prop-sne-r-c} we have $(\forall \nnn)$ 
							$
							-|y_n - x_n| + 1 \le Ty_n - Tx_n < 0
							$.
							That is  $(\forall \nnn)$ $|x_n - y_n| - |Tx_n - Ty_n|  \ge 1$, and this contradicts \cref{prop-sne-r-c}. This completes the proof.		
						\end{proof}
						
						We now turn to the correspondence
						between uniformly monotone operators
						and \ssnonex\ operators.  	We start with the following lemma that provides a characterization 
						of uniformly monotone operators via their reflected resolvents.
						
						\begin{lemma}
							\label{prop:um:sn:i}
							The following hold:
							\begin{enumerate}
								\item
								\label{prop:um:sn:i:i}
								Suppose that	$A$ is uniformly monotone with modulus $\phi$.	
								Then
								$(\forall (x,y)\in X\times X)$
								\begin{equation}
									\label{eq:ref:res:um}
									\norm{x-y}^2-\norm{R_Ax-R_Ay}^2\ge 4\phi\big(\tfrac{1}{2}\norm{(x-y)+(R_Ax-R_Ay)}\big)
									=4\phi\big( \norm{J_Ax-J_Ay}\big).
								\end{equation}	
								\item
								\label{prop:um:sn:i:ii}
								Suppose that there exists a modulus function  
								$\phi$
								such that 	
								$(\forall (x,y)\in X\times X)$
								\begin{equation}
									\label{eq:assmp:RA:mod}
									\norm{x-y}^2-\norm{R_Ax-R_Ay}^2\ge \phi\big(\norm{(x-y)+(R_Ax-R_Ay)}\big).
								\end{equation}		
								Then $A$ is uniformly monotone with a
								modulus $\tfrac{1}{4}\phi\circ (2(\cdot))$.
							\end{enumerate}
						\end{lemma}
						\begin{proof}
							\cref{prop:um:sn:i:i}: It follows from \cref{eq:gr:A:RA} that
							$	\gra  A=\menge{\tfrac{1}{2}(x+R_Ax,x-R_Ax)}{x\in X}$.
							Now combine this with  \cref{eq:def:um}. %and \cref{eq:gr:A:RA:i}.
							\cref{prop:um:sn:i:ii}:
							Let $\{(x,x^*),(y,y^*)\}\subseteq \gra A$
							and observe that \cref{eq:gr:A:RA} implies that 
							\begin{equation}
								\label{eq:gra:RA:opp}
								\{(x+x^*,x-x^*),(y+y^*,y-y^*)\}\subseteq \gra R_A.
							\end{equation}		
							Now
							\begin{subequations}
								\label{eq:un:RA}
								\begin{align}
									&\quad\scal{x-y}{x^*-y^*}
									\\
									&=\tfrac{1}{4}\scal{x+x^*-(y+y^*)+x-x^*-(y-y^*)}{
										x+x^*-(y+y^*)-(x-x^*-(y-y^*))}
									\\
									&=\tfrac{1}{4}(\norm{x+x^*-(y+y^*)}^2-\norm{x-x^*- (   y-y^*)}^2)
									\ge \tfrac{1}{4}\phi (2\norm{x-y}),
									\label{eq:un:RA:c}
								\end{align}	
							\end{subequations}
							where the inequality in \cref{eq:un:RA:c}
							follows from 
							combining \cref{eq:assmp:RA:mod} 
							applied with $(x,y)$ replaced 
							with $(x+x^*,y+y^*)$
						 and  \cref{eq:gra:RA:opp}.
							The proof is complete.
						\end{proof}

						\begin{proposition}
							\label{prop:um:sne}
							%			Let $A\colon X\rras X$ be maximally monotone.
							Consider the following statements:
							\begin{enumerate}
								\item
								\label{prop:um:sne:i}
								$A$ is uniformly monotone.
								\item	
								\label{prop:um:sne:ii}
								$-R_A$ is \ssnonex.
								\item	
								\label{prop:um:sne:iii}
								$-R_A$ is strongly nonexpansive.
							\end{enumerate}
							Then 
							\cref{prop:um:sne:i}$\siff$\cref{prop:um:sne:ii}$\RA$\cref{prop:um:sne:iii}.
						\end{proposition}
						
						\begin{proof}
							\cref{prop:um:sne:i}$\RA$\cref{prop:um:sne:ii}: 
							Let $\phi$ be a modulus function for $A$.
							Recalling \cref{eq:corresp} we have $-R_A$,
							as is 	$R_A$, is nonexpansive.
							Let $(x_n-y_n)_\nnn$ be a sequence in $X$
							and suppose that $\norm{x_n-y_n}^2-\norm{R_Ax_n-R_Ay_n}^2\to 0$.
							Combining this with \cref{prop:um:sn:i}, we learn that
							$0\le \phi(\norm{(x_n+R_Ax_n)-(y_n+R_Ay_n)})\to 0$.
							Since $\phi$ is increasing and vanishes only at $0$
							we must have  $(x_n+R_Ax_n)-(y_n+R_Ay_n)\to 0$,
							hence 
							$-R_A$ is \vsne. 
							
							\cref{prop:um:sne:ii}$\RA$\cref{prop:um:sne:i}: 
							Suppose $A$ is not uniformly monotone but $-R_A$ is \vsne. 
							Because $A$ is not uniformly monotone, there exist sequences 
							$(a_n,  a^*_n)_\nnn$ and $ (b_n, b^*_n)_\nnn$
							in $ \gra A$ and $\epsilon > 0$ such that
							\begin{equation} 
								\label{prop:vsne:um:i}
								\langle a_n - b_n, a^*_n -b^*_n \rangle \to 0
								\ \ 
								\mbox{but}\ \ (\forall \nnn) \|a_n - b_n\| \ge \epsilon.
							\end{equation}
							Set $(\forall \nnn)$ $x_n=a_n+a^*_n$ and $y_n=b_n+b^*_n$
							and observe that Minty's parametrization of $\gra A$ implies that
							$(\forall \nnn)$ 
							$(a_n,a^*_n,b_n,b^*_n)
							=\tfrac{1}{2}(x_n+R_Ax_n,x_n-R_Ax_n,y_n+R_Ay_n,y_n-R_Ay_n)$.
							Therefore 
							\begin{subequations}
								\begin{align}
									&\quad\|x_n - y_n\|^2 - \|-R_A x_n -(- R_A y_n)\|^2 
									\nonumber
									\\
									&=\|x_n - y_n\|^2 - \|R_A x_n - R_A y_n\|^2
									\\
									&=\scal{x_n - y_n+(R_A x_n - R_A y_n)}{x_n - y_n-(R_A x_n - R_A y_n)}
									\\
									&=\scal{x_n + R_A x_n -(y_n+ R_A y_n)}{x_n - R_A x_n -(y_n- R_A y_n)}
									\\
									&=4\scal{a_n - b_n}{ a^*_n -b^*_n}\to 0,
									\label{prop:vsne:um:e}
								\end{align}
								\label{prop:vsne:um:ii} 
							\end{subequations}	
							where the limit in \cref{prop:vsne:um:e} 
							follows from \cref{prop:vsne:um:i}.
							Because $-R_A$ is \vsne\ \cref{prop:vsne:um:ii} implies 
							\begin{equation}
								\label{prop:vsne:um:iii}		
								a_n-b_n=\tfrac{1}{2}((x_n + R_A x_n) - (y_n + R_A y_n))
								=\tfrac{1}{2}((x_n - y_n) - (-R_A x_n -(-R_A y_n)) )\to 0.
							\end{equation}	
							This contradicts (\ref{prop:vsne:um:i}), hence $A$ is uniformly monotone.
							\cref{prop:um:sne:ii}$\RA$\cref{prop:um:sne:iii}:
							Apply \cref{lem:une:sne} with $T$ replaced by $-R_A$.
						\end{proof}

						\begin{corollary} 
							\label{cor-real-line-eq}
							Suppose that $X=\RR$.
							The following are equivalent.
							\begin{enumerate}
								\item $A$ is uniformly monotone.
								\item $-R_A$ is strongly nonexpansive.
								\item $-R_A$ is \vsne.
							\end{enumerate}
						\end{corollary}
						
						\begin{proof}
							Combine \cref{prop:um:sne} and 	\cref{prop-sne-r}.
						\end{proof}	
						
						\begin{example}
							\label{ex:umoce} 
							Suppose that $X=\mathbb{R}^2$.
							Set $a_0=0$ and set $(\forall m \ge 1)$ 
							\begin{subequations}
								\begin{align}
									a_m&=2^{m+1}-2\\
									w_m&=\tfrac{1}{\sqrt{4^m+1}}(2^m,1)\\
									K_m&=({4^m-4^{-m}})^{1/2}\\
									\beta_m&=\tfrac{1}{2^m}K_m\\
									D_m&=\menge{(x,y)}{x\le a_m}.
								\end{align}
							\end{subequations}
							Let 
							\begin{equation}
								T(x,y)\colon \RR^2\to \RR^2\colon (x,y)\mapsto 
								\begin{cases}
									(0,0),&x\le 0;
									\\
									\sum_{j=0}^{m-1}K_jw_j+(\tfrac{x-a_{m-1}}{2^m})K_mw_m,& x\in [a_{m-1},a_m], m\ge 1. 
								\end{cases}	
							\end{equation}	
							Set $\widetilde{A}=\big(\tfrac{\Id-T}{2}\big)^{-1}-\Id$.

							Then the following hold:
							\begin{enumerate}
								\item
								\label{ex:umoce:0}
								$(\forall m \in  \mathbb{N})$ 
								$T_{|D_m}$ is a contraction with a constant $\beta_m$.
								\item
								\label{ex:umoce:i} 
								$T$ is strongly nonexpansive, hence nonexpansive.
								\item
								\label{ex:umoce:ii}
								There exist sequences $(x_n)_\nnn, (y_n)_\nnn $ in $ \RR^2$ satisfying
								\begin{equation} 
									\label{e:umoce}
									\|x_n - y_n\|^2 - \|T x_n - Ty_n\|^2 \to 0, \ \ 
									\mbox{nevertheless}\ \ (x_n - y_n) - (Tx_n - Ty_n) \not\to 0.
								\end{equation}
								Consequently, $T$ is not \ssnonex.
								\item
								\label{ex:umoce:iii:0}
								$T=-R_{\widetilde{A}}$.
								\item
								\label{ex:umoce:iii}
								$\widetilde{A}$ is maximally monotone.
								\item
								\label{ex:umoce:iv} 
								$\widetilde{A}$ is \emph{not} uniformly monotone.
							\end{enumerate}
						\end{example}
						\begin{proof}
							Observe that $(w_m)_{m\in \NN}$ is a sequence of unit vectors whose positive slopes are strictly decreasing 
							to $0$, that $(\beta_m)_{m\in \NN}$ is a sequence of strictly increasing real numbers  in $\left[0,1\right[$
							and that $(K_m)_{m\in \NN}$ is a sequence of strictly increasing real numbers with $K_m\to +\infty$. 
							
							\cref{ex:umoce:0}:
							Let $m\in \NN$.
%							 and let $\{u,v\}\subseteq D_m$.
							Observe that if $m=0 $ the conclusion is obvious.
							Therefore, we assume $m\ge 1$.
							Let $\{u,v\}\subseteq D_m$, with $u=(u_1,u_2)$ and $v=(v_1,v_2)$ and let $\{r,s\}\subseteq \{1, \ldots,m\}$
							be such that $u_1\in [a_{r-1}, a_r]$ and $v_1\in [a_{s-1}, a_s]$.
							Without loss of generality we may and do assume that $u_1\le v_1$ and hence $r\le s\le m$.
							%We proceed by proving the following claim.
							%
							%\textsc{Claim~1:}
							If $r=s$ then the definition of $T$ implies that  
							\begin{equation} 
								\label{e:umoce:i}
								Tu - Tv=\beta_r(u_1-v_1)w_r.
							\end{equation}
							Consequently, we have 
							\begin{equation} 
								\label{e:cont:oneint}
								\norm{Tu-Tv}
								=\beta_r\abs{u_1-v_1}\le \beta_r \norm{u-v}. 	
							\end{equation}
							Observe that $(\forall (u_1,u_2)\in \RR^2)$ we have 
							$T(u_1,u_2)=T(u_1,0)$.
							Moreover, let $ k\ge 1$. Applying \cref{e:umoce:i} with 
							$(u,v)$ replaced by $((a_{k-1},0),(a_{k},0))$ yields:
							\begin{equation}
								\label{e:am:am1}
								{T(a_k,0)-T(a_{k-1},0)}
								=\beta_k(a_{k}-a_{k-1})w_k=K_kw_k.
							\end{equation}	
							Using the triangle inequality, \cref{e:umoce:i}
							and \cref{e:am:am1} 
							we have
							\begin{subequations}
								\begin{align}
									\norm{Tu-Tv}
									&=\norm{T(v_1,v_2)-T(u_1,u_2)}=\norm{T(v_1,0)-T(u_1,0)}
									\\
									&=	\norm{T(v_1,0) - T(a_{s-1},0) +T(a_{s-1},0)-T(a_{s-2},0)+\ldots
										\nonumber
										\\
										&
										\quad	-T(a_{r},0)+T(a_{r},0)    -T(u_1,0)}
									\\
									&\le 	\norm{T(v_1,0) - T(a_{s-1},0)}+\norm{T(a_{s-1},0)-T(a_{s-2},0)}+\ldots
									\nonumber
									\\
									&\quad +\norm{T(a_{r},0)    -T(u_1,0)}
									\\
									&=\beta_s(v_1-a_{s-1})+\beta_{s-1}(a_{s-1}-a_{s-2})+\ldots+\beta_{r+1}(a_{r+1}-a_{r})+\beta_r(a_r-u_1)
									\\
									&\le\beta_m(v_1-a_{s-1})+\beta_{m}(a_{s-1}-a_{s-2})+\ldots+\beta_{m}(a_{r+1}-a_{r})+\beta_m(a_r-u_1)
									\\
									&=\beta_m(v_1-u_1)
									\le \beta_m\norm{u-v}.
								\end{align}	
							\end{subequations}

							\cref{ex:umoce:i}:
							Suppose $(x_n)_\nnn$, and $ (y_n)_\nnn$ are sequences in $ \RR^2$ such that
							\begin{equation} \label{e:umoce:iii}
								\|x_n - y_n\| - \|Tx_n - Ty_n\| \to 0  \ \ \mbox{and} \ \ (x_n - y_n)_\nnn \ \mbox{is bounded}.
							\end{equation}
							Let us denote $x_n = (x_{n,1},x_{n,2})$, $y_n = (y_{n,1},y_{n,2})$. 
							We proceed by proving the following claims:
							
							\textsc{Claim~1}:
							The sequences $(x_{n,1})_\nnn$, and $ (y_{n,1})_\nnn$ are unbounded.
							\\
							Indeed, suppose for eventual contradiction that one of the sequences $(x_{n,1})_\nnn$, and $ (y_{n,1})_\nnn$ is bounded.
							The boundedness of $(x_n - y_n)_\nnn$ implies 
							that both  sequences $(x_{n,1})_\nnn$ and $ (y_{n,1})_\nnn$ must be bounded.	
							Indeed, without loss of generality we may and do assume that $(x_n-y_n)\not\to 0$.
							Let $m\ge 1$ be such that $(\forall \nnn)$ $\max \{x_{n,1},y_{n,1}\}\le a_m$.
							Observe that \cref{ex:umoce:0} implies that 
							$\norm{Tx_n-Ty_n}\le \beta_m\norm{x_n-y_n}$. Hence,
							\begin{equation}
								0<(1-\beta_m)\norm{x_n-y_n}\le \norm{x_n-y_n}-\norm{Tx_n-Ty_n}\to 0.
							\end{equation}	
							That is, $\norm{x_n-y_n}\to 0$, which is absurd.
							Therefore,   the sequences $(x_{n,1})_\nnn$ and $ (y_{n,1})_\nnn$ are unbounded
							as claimed.
							After passing to a subsequence and relabelling if necessary we conclude that
							\begin{equation}
								x_{n,1}\to \infty \quad \text{and}	\quad y_{n,1}\to \infty.
							\end{equation}
							Because $(x_n - y_n)_\nnn$ is bounded and the slopes of the
							vectors $(w_n)_\nnn$ go to $0$ as $n \to \infty$.  
							\textsc{Claim~2:}
							\begin{equation}
								\text{$(\forall \nnn)$ $Tx_n - Ty_n = (c_n(x_{n,1} - y_{n,1}), d_n)$ where $c_n \to 1^{-}$  and $d_n \to 0$. }
							\end{equation}
							To verify \textsc{Claim~2},  let $M > 0$ be such that
							$(\forall \nnn)$ $\|x_n - y_n\| \le M$. 
							it follows from 
							\cref{e:umoce:i} that 
							\begin{equation} 
								\label{e:umoce:ii:i}
								Tu - Tv=\tfrac{K_r}{{\sqrt{4^r+1}}}({u_1-v_1})\big(1, \tfrac{1}{2^r}\big).
							\end{equation}
							In view of \cref{e:umoce:ii:i}
							and \cref{e:cont:oneint}
							we learn that
							for $n$  and $j$ such that $\min\{x_{n,1},y_{n,1}\} \ge a_{j-1}$, 
							it follows  that
							\begin{equation}
								\label{e:umoce:iv}
								|d_n| \le 2^{-j}M \qquad \mbox{while} \qquad c_n \ge \frac{K_j }{{\sqrt{4^j+1}}}= \left(\frac{4^j - 4^{-j}}{4^j + 1}\right)^{1/2}.
							\end{equation}
							Note  it is possible that the interval with endpoints $x_{n,1}$ and $y_{n,1}$ may intersect more than one interval $[a_{r-1},a_r]$.  
							However, $2^{-n}$ decreases as $n$ increases and 
							$K_n/{\sqrt{4^n+1}}$ increases as $n$ increases in  \cref{e:umoce:ii:i}
							and so \cref{e:umoce:iv} remains valid in this case as well 
							since $a_{j - 1} \le \min\{x_{n,1},y_{n,1}\}$.
							This verifies \textsc{Claim~2}.
							
							\textsc{Claim~3:}
							\begin{equation}
								|x_{n,2} - y_{n,2}| \to 0. 
							\end{equation}
							
							Indeed, set $(\forall \nnn) $ $\overline{x}_n=(x_{n,1},0)$ and $\overline{y}_n=(y_{n,1},0)$
							and note that $\norm{\overline{x}_n-\overline{y}_n}\le \norm{x_n-y_n}$
							and that $(T\overline{x}_n,T\overline{y}_n)=(Tx_n,Ty_n)$.
							Therefore, the nonexpansiveness of $T$ and \cref {e:umoce:iii} imply
							\begin{equation}
								0\le 	\norm{\overline{x}_n-\overline{y}_n}-\norm{T\overline{x}_n-T\overline{y}_n}
								=	\norm{\overline{x}_n-\overline{y}_n}- \|Tx_n - Ty_n\|
								\le 
								\|x_n - y_n\| - \|Tx_n - Ty_n\| \to 0.
							\end{equation}	
							That is 
							\begin{equation}	
								\label{e:aux:seq}
								\norm{\overline{x}_n-\overline{y}_n}- \|Tx_n - Ty_n\|\to 0.
							\end{equation}	
							Subtracting \cref{e:umoce:iii} and \cref{e:aux:seq} yields
							\begin{equation}	
								\norm{ {x}_n- {y}_n} - 	\norm{\overline{x}_n-\overline{y}_n} \to 0.
							\end{equation}	
							Consequently we have 
							\begin{subequations}	
								\begin{align}
									|x_{n,2} - y_{n,2}| ^2&=	\norm{ {x}_n- {y}_n}^2 - 	\norm{\overline{x}_n-\overline{y}_n}^2
									\\
									&=(	\norm{ {x}_n- {y}_n}- 	\norm{\overline{x}_n-\overline{y}_n})(	\norm{ {x}_n- {y}_n}+	\norm{\overline{x}_n-\overline{y}_n})  \to 0.
								\end{align}
							\end{subequations}	
							Combining \textsc{Claim~2}
							and \textsc{Claim~3} 
							we learn that 
							\begin{equation}
								(x_n-y_n)-(Tx_n-Ty_n)\to (0,0),
							\end{equation}	
							hence
							$T$ is strongly nonexpansive.
							
							\cref{ex:umoce:ii}:
							%Finally, let us establish the claim in (\ref{e:umoce}). 
							Set $(\forall \nnn)$  $x_n = (a_{n+1},0)$ and $y_n = (a_{n}, 0)$. Observe that
							$x_n - y_n = (2^n,0)$ while $Tx_n - Ty_n = K_n w_n = \tfrac{K_n}{\sqrt{4^n+1}}( 2^n,1)$. 
							Therefore, $\|x_n - y_n\|^2 = 4^n$
							and $\|Tx_n - Ty_n\|^2 = \| K_n w_n\|^2 = K_n^2=4^n - 4^{-n}$. Hence
							\begin{equation}
								\|x_n - y_n\|^2 - \|T x_n - Ty_n\|^2 \to 0
							\end{equation}	
							However, %since $\tfrac{K_n}{\sqrt{4^n+1}} \to 1$, it follows that 
							\begin{equation}
								(x_n - y_n) - (Tx_n - Ty_n)
								=\Big(\big(1-\tfrac{K_n}{\sqrt{4^n+1}}\big)2^n,-\tfrac{K_n}{\sqrt{4^n+1}}\Big) 
								\to (0,-1)\neq (0,0).
							\end{equation}
							and so (\ref{e:umoce}) holds. 
							
							\cref{ex:umoce:iii:0}:
							This is clear.
							
							\cref{ex:umoce:iii}:
							Combine
							\cref{ex:umoce:i} and \cite[Corollary~23.9~and~Proposition~4.4]{BC2017}.
							
							\cref{ex:umoce:iv}:
							Combine
							\cref{ex:umoce:ii},
							\cref{ex:umoce:iii:0} and \cref{prop:um:sne}.
						\end{proof}
						\begin{figure}[H]
							\begin{center}
								\includegraphics[scale=0.6]{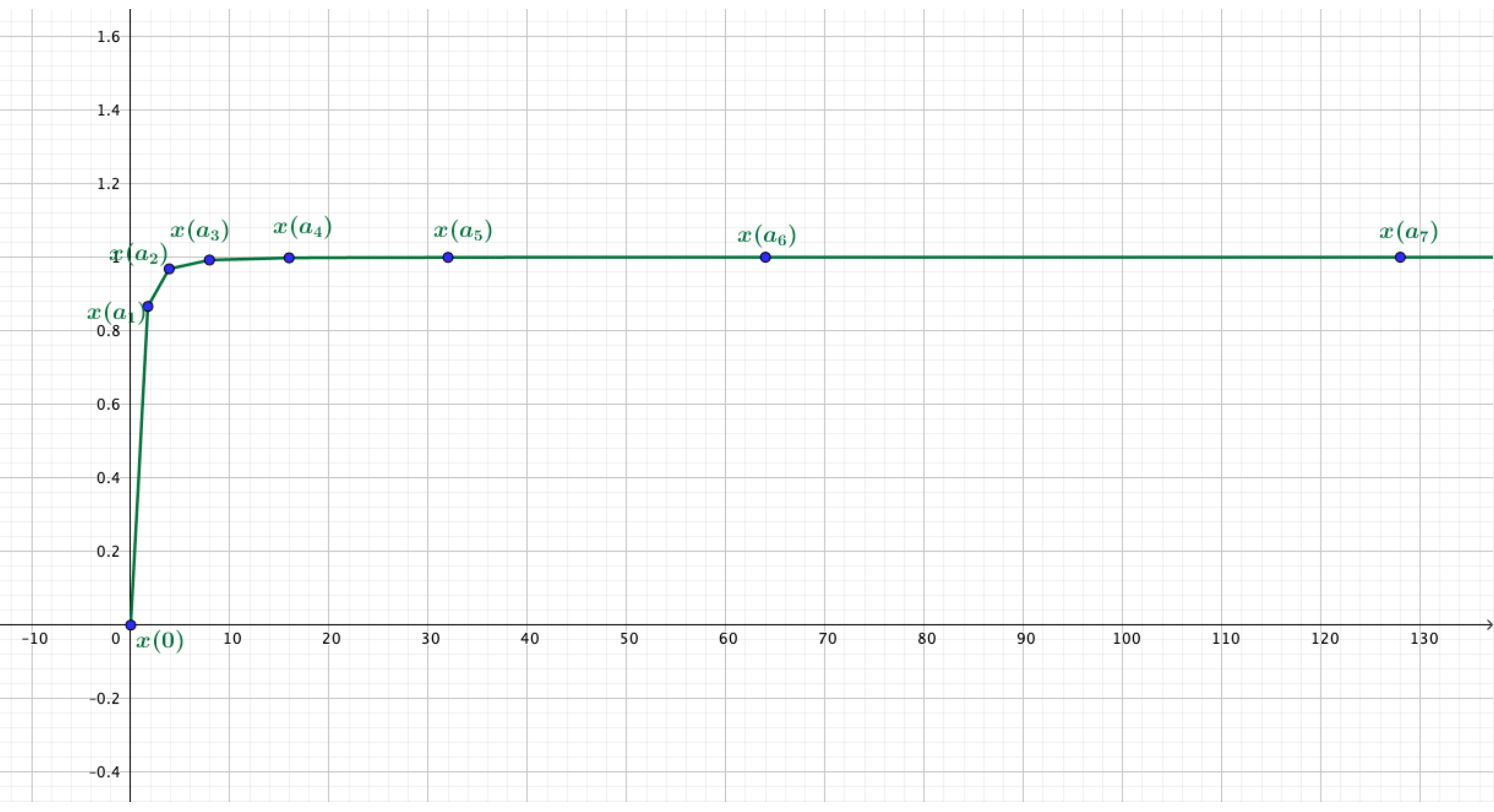}
							\end{center}	
							\caption{A \texttt{GeoGebra} snapshopt illustrating the operator 
								$T$ in \cref{ex:umoce}. Here $x(0)\coloneqq T(x,y), x\le 0, y\in \RR$
								and  $x(a_i)\coloneqq T(a_i,y)$ where $y\in \RR$, $i\in\{1,2,\ldots, 7\} $.}
						\end{figure}

						\section{Properties of Uniformly Monotone Operators}
						\label{sec:4}
						The main goals of this section are to establish that uniformly 
						monotone operators are surjective, have unique zeros and 
						have uniformly continuous inverse operators.
						
						The following result is motivated by Z\u{a}linescu's important 
						result \cite[Proposition 3.5.1]{Za02} concerning the growth rate of the modulus of a uniformly convex function. 
						
						\begin{proposition} 
							\label{prop-uni-mon-mod}
							Let $Y$ be a Banach space
							and suppose $B\colon Y \rras {Y^*}$ is uniformly monotone\footnote{Let $Y$ be a Banach space. An operator $B \colon
								Y \rras {Y^*}$ is {\em uniformly monotone} with a modulus 
								$\phi\colon  \RR_+ \to [0,+\infty]$ if $\phi$ is increasing, vanishes only at $0$ and 
								$(\forall (x,x^*)\in \gra B)$ $(\forall (y,y^*)\in \gra B)$
								$
								\langle x - y, x^*-y^*\rangle\ge \phi(\|x- y\|)  .
								$} 
							with convex domain. 
							Then $B$ has a supercoercive modulus $\phi$ that satisfies 
							the following property.
							\begin{equation} 
								\label{eq-ums}
								\mbox{For each} \ \epsilon > 0 \ \  (\exists\beta_\epsilon > 0) \ \ 
								\mbox{so that}\ \ \phi(t) \ge \beta_\epsilon t^2 \ \ 
								\mbox{whenever} \ t \ge \epsilon, \ \mbox{in particular}\ 
								\liminf_{t \to \infty} \frac{\phi(t)}{t^2} > 0.
							\end{equation}
						\end{proposition}
						
						\begin{proof}
							Because $B$ is uniformly monotone we fix 
							$\alpha > 0$ so that
							$(\forall (x,x^*)\in \gra B)$ $(\forall (y,y^*)\in \gra B)$
							\begin{equation} 
								\label{eq-uma}
								\langle y - x, y^* - x^* \rangle \ge \alpha \ \ \mbox{whenever}  \  \|y - x\| \ge 1.
							\end{equation}
							We claim that:
							\begin{equation} 
								\label{eq-umc}
								\langle y - x, y^* - x^* \rangle \ge \tfrac{ \alpha}{4}\|x-y\|^2 \ \mbox{whenever}  \ 
								\|y - x\| \ge 1,
								(x,x^*)\in \gra B, \ (y,y^*)\in \gra B.
							\end{equation}
							%		Using induction we will show for $k\in \{0,1,2,3,\ldots\}$ that
							To verify \cref{eq-umc} 
							it is sufficient to show  that $(\forall k\in \NN) $ we have  
							\begin{equation} 
								\label{eq-umb}
								\langle y - x, y^* - x^* \rangle \ge 2^{2k} \alpha \ \mbox{whenever}  \ 
								\|y - x\| \ge 2^k, \ (x,x^*)\in \gra B, \ (y,y^*)\in \gra B. 
							\end{equation}
							We proceed by induction on $k$. Indeed, \cref{eq-uma} verifies 
							the base case at $k=0$.
							Now suppose that \cref{eq-umb} is true for some
							$k = n$. Suppose that $\ (x,x^*)\in \gra B, \ (z,z^*)\in \gra B$ 
							satisfy $\|z - x\| \ge 2^{n+1}$. 
							Let $y = (x+z)/2$ and observe that 
							$y- x= z- y = (z - x)/2 $, hence $\norm{y- x}= \norm{z- y}\ge 2^n$. 
							Because $\dom B$ is convex 
							we have $y \in \domai B$ and so we can pick  $y^* \in By$.   
							It follows from the
							inductive hypothesis that
							\begin{subequations}
								\begin{align}
									2^{2n} \alpha 
									&\le \langle y - x, y^* - x^* \rangle
									= \left\langle \tfrac{z-x}{2} , y^* - x^* \right\rangle
									\label{se:ine:1}
									\\
									2^{2n} \alpha &\le \langle z -y, z^* - y^*\rangle 
									=  \left\langle \tfrac{z-x}{2} , z^* - y^* \right\rangle
									\label{se:ine:2}
								\end{align}
							\end{subequations}
							Adding \cref{se:ine:1} 
							and \cref{se:ine:2} yields $\ds \left\langle \tfrac{z-x}{2}, z^* - x^*\right\rangle \ge 2 (2^{2n})\alpha$ 
							and consequently
							\begin{equation} 
								\langle z - x, z^* - x^*\rangle \ge 4( 2^{2n})\alpha =  2^{2(n+1)} \alpha,
							\end{equation}
							which proves (\ref{eq-umc}).
							We now establish (\ref{eq-ums}). 
							Let $\epsilon >0$. In view of \cref{eq-umc} 
							if $\epsilon \ge 1$ we set $\beta_\epsilon=\tfrac{\alpha}{4}$.
							Now suppose that  $0 < \epsilon < 1$.
							The uniform monotonicity of $B$ implies that
							$(\exists \alpha_\epsilon > 0)$ such  that
							\begin{equation} 
								\label{eq-umd}
								\langle y - x, y^* - x^* \rangle \ge \alpha_\epsilon \ \ \mbox{whenever}  \  \|y - x\| \ge \epsilon,
								(x,x^*)\in \gra B, \ (y,y^*)\in \gra B.
							\end{equation}
							In view of \cref{eq-umc}
							the conclusion follows by setting 
							$\beta_\epsilon = \min\{\alpha/4,\alpha_\epsilon\}$ for $0 < \epsilon  < 1$, and $\beta_\epsilon = \alpha/4$ for $\epsilon \ge 1$. 
						\end{proof}
						
					Analogous to the concept of Lipschitz  for large distances
					 in \cite{BL2000}, we introduce the following definition.
						\begin{definition}
							\label{defn:in-the-large-cont}
							We say that $T\colon X \to X$ is 
							a {\em contraction for large distances} if for each $\epsilon > 0$, 
							$(\exists K_\epsilon < 1)$ so that 
							$\|Tx - Ty\| \le K_\epsilon\|x - y\|$  whenever $(x,y) \in X\times X$ 
							satisfy $\|x - y\| \ge \epsilon$.
						\end{definition}
					
					\begin{remark}
					Contractions for large distances are not a new concept. 
					In fact, they coincide with an unnamed class of nonexpansive maps 
					introduced by Rakotch in  
					\cite[Definition~2]{Rak62}, and have been referred to as 
					\emph{contractive}  by others, including, for example, Reich \& Zaslavski in \cite{ReiZas00}.
					 However, in \cite{Rak62} the term contractive referred to the more general class of strictly 
					 nonexpansive mappings. 
					\end{remark}	
						
						\begin{lemma}
							\label{lem-inv-res} 
							Suppose $A$ is uniformly monotone, then the following hold:
							\begin{enumerate}
								\item
								\label{lem-inv-res:a} 
								$J_{A^{-1}}$ 
								is uniformly monotone with a supercoercive modulus.
								\item
								\label{lem-inv-res:b} 
								For each $\epsilon > 0$ $(\exists\beta_\epsilon \in \left ]0,1\right])$
								such that
								$(\forall (x,y)\in X\times X)$ satisfying $\|x  - y \| \ge \epsilon$ we have 
								\begin{equation} 
									\label{eq-ir-s}
									\scal{J_Ax - J_Ay}{J_{A^{-1}}x - J_{A^{-1}}y}
									+ \|J_{A^{-1}}x - J_{A^{-1}}y\|^2 \ge \beta_\epsilon \norm{x-y}^2.
								\end{equation}
								\item
								\label{lem-inv-res:ab} 
								$J_{A^{-1}}$ is surjective.		
								\item
								\label{lem-inv-res:c} 
								$J_{A}$ is a contraction for large distances.	
							\end{enumerate}
						\end{lemma}
						
						\begin{proof}
			Let $\epsilon > 0$ and let $\phi$ be a modulus function for
							$A$.  Set 
							\begin{equation}
								\label{eq-ir-alp}
								\alpha=\alpha(\epsilon) = \min\{\phi(\epsilon/2), \epsilon^2/4\}.
							\end{equation}
							Then $\alpha > 0$. 
							\cref{lem-inv-res:a}\&\cref{lem-inv-res:b}:
	Let $(x,y)\in X\times X$ 
							be such that 
							$\|x -y\| \ge \epsilon$. 
							The triangle inequality and \cref{thm:minty} imply that  
							\begin{equation} 
								\label{eq-ir-c}
								\max \{\|J_{A}x - J_{A}y\|,\|J_{A^{-1}}x - J_{A^{-1}}y\|\} \ge \epsilon/2.
							\end{equation} 
							Therefore we obtain
							\begin{subequations}
								\begin{align} 
									\scal{x  - y}{ J_{A^{-1}}x - J_{A^{-1}}y} 
									&= \scal{J_{A}x - J_{A}y+J_{A^{-1}}x - J_{A^{-1}}y}{ J_{A^{-1}}x - J_{A^{-1}}y} 
									\\
									&=\scal{J_{A}x - J_{A}y}{J_{A^{-1}}x - J_{A^{-1}}y}
									+ \|J_{A^{-1}}x - J_{A^{-1}}y\|^2  
									\\
									&\ge \phi(\|J_{A}x - J_{A}y\|) + \|J_{A^{-1}}x - J_{A^{-1}}y\|^2 \ge \alpha.    
								\end{align}
								\label{eq-ir-b}
							\end{subequations}
							It follows from (\ref{eq-ir-b})  that $J_{A^{-1}}$ is uniformly monotone.
							Combining this  with \cref{prop-uni-mon-mod} applied with $(Y, A)$
							replaced by $(X, J_{A^{-1}})$ and the fact that 
							$\dom J_{A^{-1}}=X$ we learn that 
							$J_{A^{-1}}$ is uniformly monotone with a supercoercive 
							modulus that satisfies \cref{eq-ums}. The claim that 
							$\beta_\epsilon \le1$ is a direct consequence 
							of the nonexpansiveness of $J_{A^{-1}}$
							and the monotonicity of $A$
							in view of \cref{eq-ir-s}. 
							
							\cref{lem-inv-res:ab}: 
							Combine \cref{lem-inv-res:a}  
							and \cite[Proposition~22.11(ii)]{BC2017}.
							\cref{lem-inv-res:c}:
							Indeed, using \cref{lem-inv-res:b} and the firm 
							nonexpansiveness  of
							$J_A$ there exists $(\exists\beta_\epsilon \in \left ]0,1\right])$
							such that
							\begin{subequations}
								\begin{align} 
									0	\le \norm{J_{A}x - J_{A}y}^2&\le \scal{x  - y}{ J_{A}x - J_{A}y} 
									= \norm{x-y}^2-\scal{x-y}{ J_{A^{-1}}x - J_{A^{-1}}y} 
									\\
									&\le  \norm{x-y}^2-\phi(\norm{x-y}) \le (1-\beta_\epsilon)    \norm{x-y}^2,
								\end{align}
								\label{eq-ir-e}
							\end{subequations}	
							and the conclusion follows.
						\end{proof}
						
						The previous results lead to nice consequences for uniformly monotone 
						operators which we next state as the main result of this section.

						\begin{theorem} 
							\label{thm-umg} 
							Suppose $A\colon X \rras X$ is uniformly monotone. 
							Then the following hold.
							\begin{enumerate}	
								\item 
								\label{thm-umg:i} 
								$A$ satisfies the growth condition
								\begin{equation}
									\label{e:growth}
									\lim_{\|x-y\| \to \infty} \inf_{\substack{(x,x^*)\in \gra A\\(y,y^*)\in \gra A}} \frac{\|x^* - y^*\|}{\|x - y\|} > 0.
								\end{equation}
								\item 
								\label{thm-umg:ii} 
								$A$ is surjective.
								\item 
								\label{thm-umg:iii} 
								$A$ has a unique zero.
								\item 
								\label{thm-umg:iv} 
								$A^{-1}$ is uniformly continuous.
							\end{enumerate}
						\end{theorem}
						
						\begin{proof} 
							\cref{thm-umg:i}:
							Suppose for eventual contradiction that 
							there exist sequences  $(x_n,x_n^*)_\nnn$, $(y_n,y_n^*)_\nnn$ in $\gra A$
							such that $\|x_n - y_n\| \to \infty$ but
							\begin{equation}
								\lim_{n \to \infty} \frac{\|x_n^* - y_n^*\|}{\|x_n - y_n\|} = 0.
							\end{equation}
							Indeed, write $\|x_n^* - y_n^*\| = a_n\|x_n - y_n\|$ where $a_n \to 0^+$. 
							Then using \cref{eq-ir-s}, we have $\beta > 0$ such  that 
							\begin{equation}
								(\forall \nnn)\quad a_n\|x_n - y_n\|^2 
								+ a_n^2\|x_n - y_n\|^2 \ge \beta(\|x_n -y_n\|^2 + a_n^2\|x_n - y_n\|^2),
							\end{equation}
							where on the right side, we drop the inner product from \cref{eq-ir-s} 
							because it is nonnegative by monotonicity of A. 
							But the above inequality is impossible since $a_n \to 0^+$
							while $\beta>0$ is fixed.
							Hence, \cref{e:growth} holds.

							\cref{thm-umg:ii}: 
							Combine \cref{lem-inv-res}\cref{lem-inv-res:ab}
							and the fact that 
							$\ran A=\dom A^{-1}=\dom (\Id+A^{-1})=\ran (\Id+A^{-1})^{-1}=\ran J_{A^{-1}}$.
							
							\cref{thm-umg:iii}:  
							It follows from  \cref{thm-umg:ii} that $\zer A\neq \fady$.
							Moreover, $A$ is strictly monotone, hence $\zer A$ is at most a singleton
							by, e.g., \cite[Proposition~23.35]{BC2017}.
							Altogether, $A$ possess a unique zero.

							\cref{thm-umg:iv}:  
							Let $\epsilon > 0$. In view of \cref{thm-umg:i}
							choose $K > 0$ 
							such that $(\forall (x,x^*)\in \gra A)$ $(\forall (y,y^*)\in \gra A)$
							\begin{equation}
								\label{eq-umg-b}
								\|x - y\| \ge K\RA	\|x^* - y^*\| > \epsilon. 
							\end{equation} 
							Let $\alpha = \phi(\epsilon)$ where $\phi$ is a modulus function for $A$. 
							Then $\alpha > 0$ and $(\forall (x,x^*)\in \gra A)$ $(\forall (y,y^*)\in \gra A)$
							\begin{equation} 
								\label{eq-umg-c}
								\|x - y\| \ge \epsilon \RA \scal {x - y}{ x^* - y^*}\ge \alpha .
							\end{equation}
							Now choose $\delta = \min\{\alpha/K, \epsilon\}$. 
							Recalling \cref{thm-umg:ii}, let $\{x^* ,y^*\} \subseteq X=\dom A^{-1}$. 
							Suppose $\|x^* - y^*\| < \delta$ and let $(x^*,x) ,(y^*,y) $
							be points in $\gra  A^{-1}$,
							equivalently; $(x,x^*) ,(y,y^*) $
							are points in $\gra  A$.
							Because $\|x^* - y^*\| < \delta \le \epsilon$, (\ref{eq-umg-b}) implies $\|x - y\| < K$. Because $\|x - y\| < K$, using
							Cauchy--Schwarz  
							we  obtain
							\begin{equation} 
								\label{eq-umg-d}
								\scal{x - y}{x^* - y^*} \le \|x-y\| \|x^* - y^*\| < K\delta  \le \alpha.
							\end{equation}
							Combining \cref{eq-umg-d} and \cref{eq-umg-c} we learn that 
							$\|x - y\| < \epsilon$. Therefore $A^{-1}$ is uniformly continuous as desired.
						\end{proof}
						\begin{remark}
							In passing we point out that 
							\cref{thm-umg}\cref{thm-umg:ii}\&\cref{thm-umg:iii} relax the assumptions 
							of 
							\cite[Proposition 22.11~and~Corollary 23.37]{BC2017}.
							Indeed,  \cref{thm-umg}\cref{thm-umg:ii}\&\cref{thm-umg:iii} 
							assume only the uniform monotonicity of $A$ and do not require supercoercivity of the  modulus.
						\end{remark}
						
						Following \cite[p. 160]{Simons2}, we will say 
						%a (possibly) set-valued operator
						that $S\colon X \rras {X}$ is {\em coercive} provided 
						that $\inf \langle x,Sx\rangle/\|x\| \to \infty$ as $\|x\| \to \infty$, where 
						we use the standard convention that the infimum of the empty set is $+\infty$. 
						Or in other words, given any $K > 0$  $(\exists M > 0)$ so that
						\begin{equation}
							\langle x,x^* \rangle \ge K\|x\| \ \mbox{whenever} \ \ \|x\| \ge M, (x,x^*)\in \gra S.
						\end{equation}				
						Neither the growth condition in \cref{thm-umg}\cref{thm-umg:i}  
						and nor coercivity implies the other for monotone operators 
						as we illustrate in \cref{ex:sc}  
						and \cref{ex-uc-dual-fail}\cref{ex-uc-dual-fail:ii:ii}\&\cref{ex-uc-dual-fail:ii}.
						
						\begin{example} 
							\label{ex:sc} 
							Let $f\colon \RR^2 \to \RR$ be defined by
							\begin{equation}
								f(\xi_1,\xi_2) = 
								\begin{cases} \xi_1^2, &\mbox{if}\;\;  0 \le \xi_1, \ 0 \le \xi_2 \le\xi_1; \\
									+\infty, &\mbox{otherwise.}
								\end{cases}
							\end{equation}
							Set $A = \partial f$. Then 
							\begin{equation}
								\label{e:sc:i}
								\lim_{\|x\| \to \infty} \inf \left\{\tfrac{\langle x^*, x\rangle}{\|x\|^2} \,\Big|\, x^* \in Ax \right\} > 0.
							\end{equation}
							Hence, $A$ is coercive. However, $A$ does not
							satisfy the growth condition in \cref{thm-umg}\cref{thm-umg:i}.   
						\end{example}
						
						\begin{proof}
							Because $f(x) \ge \frac{1}{2}{\|x\|^2}$ $(\forall x \in \RR^2)$
							it follows that (\ref{e:sc:i}) holds.
							To verify the second claim, set $(\forall \nnn)$ 
							$x_n = (n,0)$, $y_n = (n,n)$ and 
							$x_n^* =y_n^* =  (2n,0)$.
							Observe that $\{(x_n,x_n^*),(y_n,y_n^*)\}\subseteq \gra A$. 
							Consequently,
							$(\forall \nnn)$ $\tfrac{\norm{x_n^*-y_n^*}}{\norm{x_n-y_n}}=\tfrac{0}{n}=0$.
							The proof is complete.
						\end{proof}

						\begin{example} 
							\label{ex-uc-dual-fail}
							Suppose that  
							$S\colon \RR^2 \to \RR^2\colon (x_1,x_2) \to  (-x_2,x_1)$  
							is the rotator by $\pi/2$.
							Then the following properties hold.
							\begin{enumerate}
								\item
								\label{ex-uc-dual-fail:i}
								$S$ and $S^{-1}=-S$ are maximally monotone. 
								\item
								\label{ex-uc-dual-fail:i:ii}
								Both $S$ and $S^{-1}$ are isometries, hence 
								are Lipschitz continuous.
								\item
								\label{ex-uc-dual-fail:ii:ii}
								$	\lim_{\|x-y\| \to \infty} \inf \tfrac{\|Sx - Sy\|}{\|x - y\|}=1 > 0$.
								\item
								\label{ex-uc-dual-fail:i:iii}
								Both $S$ and $S^{-1}$ are
								uniformly continuous.
								\item
								\label{ex-uc-dual-fail:ii}
								Neither $S$ nor $S^{-1}$ are uniformly monotone, nor are they coercive.
								\item
								\label{ex-uc-dual-fail:iii}
								Both $J_S$ and $J_{S^{-1}}$ are strongly monotone.
							\end{enumerate}
						\end{example}
						
						\begin{proof} 
							\cref{ex-uc-dual-fail:i}:
							This is \cite[Example~22.15]{BC2017}.
							\cref{ex-uc-dual-fail:i:ii}: This is clear.
							\cref{ex-uc-dual-fail:ii:ii}\&\cref{ex-uc-dual-fail:i:iii}:
							This is a direct consequence of \cref{ex-uc-dual-fail:i:ii}.		
							\cref{ex-uc-dual-fail:ii}: 
							Indeed, $S$ is neither uniformly monotone nor coercive 
							because   $(\forall x\in \RR^2)$ $\langle Sx,x \rangle = 0$. 
							The same properties hold for $S^{-1}=-S$.  		
							\cref{ex-uc-dual-fail:iii}: 
							Observe that $J_S= (\Id+ S)^{-1}$ and so $\ds J_S = \tfrac{1}{2}(\Id-S)$
							and $(\forall x \in \RR^2)$
							$\scal{J_S x}{ x} = \frac{1}{2}\|x\|^2$.
							Hence $J_S$ is strongly monotone. 
							Similarly, one verifies that $(\forall x \in \RR^2)$
							$\scal{J_{S^{-1}} x}{ x} = \frac{1}{2}\|x\|^2$.
						\end{proof}

						\begin{remark}\
							\begin{enumerate}
								\item 
								In the convex function case, $f$ is uniformly convex 
								if and only if $f^*$ has uniformly continuous derivative 
								(see \cite[Theorem 3.5.5 and Theorem 3.5.6]{Za02}. 
								\cref{ex-uc-dual-fail}\cref{ex-uc-dual-fail:ii}
							shows that this correspondance does not hold
							for general maximally monotone operators. 
								
								\item
								\cref{ex-uc-dual-fail}\cref{ex-uc-dual-fail:ii}\&\cref{ex-uc-dual-fail:iii} 
								also shows in a strong way that if 
								$J_{A^{-1}}$ is uniformly monotone (resp. coercive), 
								it does not automatically follow that
								$A$ is uniformly monotone (resp. coercive). 
								This is in contrast to the situation for convex functions 
								where the infimal convolution of $\|\cdot\|^2$ and $f$ is 
								uniformly convex (resp. supercoercive) if and only if 
								$f$ is uniformly convex (resp. supercoercive). 
								This follows by checking the conjugate of that infimal convolution 
								is uniformly smooth  
								if and only if $f^*$ is. 
								See \cite{BC2017,BV10,Za02} for more information on this.
							\end{enumerate}
						\end{remark}

						\section{Contractions for large distances}
						\label{sec:5}
						We start with the following lemma which provides a characterization of 
						Banach contractions using averaged mappings. 
						\begin{lemma}
							\label{lem:Bcont:duality}
							Let $T\colon X\to X$.
							Then $T $ is a Banach contraction if and only if
							[$T$ is averaged and $-T$ is averaged].
						\end{lemma}	
						\begin{proof}
							$(\RA)$: Suppose that $T $ is a Banach contraction.
							Observe that $-T $ is also a  Banach contraction.
							The conclusion follows from 
							\cite[Proposition~4.38]{BC2017}.
							
							$(\LA)$: 
							Suppose that both $T$ and $-T$ are averaged. 
							It follows from \cite[Proposition~4.3(ii)]{BMW2021} 
							that $T=R_A$ and both $A$ and $A^{-1}$
							are strongly monotone operators.
							Therefore, by \cite[Corollary~4.7]{BMW12} 
							$T$ is a Banach contraction.
						\end{proof}	
						
						Before we proceed, we present the following 
						useful result.
						
						\begin{lemma}
							\label{prop:primal:dual}	
							Let $T\colon X\to X$.
							Suppose that $T$ and $-T$
							are strongly nonexpansive.
							Suppose that $(x_n-y_n)_\nnn$ is bounded
							and that $\norm{x_n-y_n}-\norm{Tx_n-Ty_n}\to 0$.
							Then $(x_n-y_n)\to 0$.
						\end{lemma}	
						\begin{proof}
							By assumption we have
							\begin{subequations}
								\begin{align}
									(x_n-y_n)-(Tx_n-Ty_n)&\to 0
									\\
									(x_n-y_n)+(Tx_n-Ty_n)&\to 0.
								\end{align}
							\end{subequations}	
							Adding the above limits yields the desired result.
						\end{proof}	
						
						In an analogy to \cref{lem:Bcont:duality} we present 
						the following result that characterizes contractions for large distances
						using either strongly nonexpansive mappings or \ssnonex\ mappings.

						\begin{proposition} 
							\label{new:pi}
							Suppose $T\colon  X \to X$ is a nonexpansive mapping. 
							Then the following are equivalent:
							\begin{enumerate}
								\item
								\label{new:pi:i}
								Both $T$ and $-T$ are \ssnonex.
								\item
								\label{new:pi:ii}
								Both $T$ and $-T$ are strongly nonexpansive.
								\item
								\label{new:pi:iii}
								$T$ is a contraction for large distances.	 
							\end{enumerate}
						\end{proposition}
						
						\begin{proof} 
							\cref{new:pi:i} $\Rightarrow$ \cref{new:pi:ii}:
							Apply \cref{lem:une:sne} to both $T$ and $-T$.
							
							\cref{new:pi:ii} $\Rightarrow$ \cref{new:pi:iii}: 
							Fix $\epsilon > 0$, and define
							$$
							\beta = \sup \left\{ \frac{\|Tx - Ty\|}{\|x-y\|}\, \Big{|} \, \ \epsilon \le \|x-y\| \le 2\epsilon\right\}.
							$$
							We claim that $\beta < 1$. Indeed, suppose
							for eventual contradiction that $\beta = 1$. 
							Then there exist  sequences  $(x_n )_\nnn$ and 
							$( y_n)_\nnn$ in $ X$  
							such that 
							\begin{equation} 
								\text{  $(\forall \nnn)$ $\epsilon \le \|x_n - y_n\| \le 2\epsilon$ 
									\;\; and\;\;  $\frac{\|Tx_n - Ty_n\|}{\|x_n-y_n\|} \to  1.$} 
							\end{equation}
							Because $\epsilon \le \|x_n - y_n\| \le 2\epsilon$, 
							we have $\|Tx_n - Ty_n\| - \|x_n - y_n\| \to  0$.
							Therefore, it follows from \cref{prop:primal:dual} that 
							$(x_n - y_n) \to 0$, which is absurd since  
							$(\forall \nnn)$
							$\|x_n - y_n\| \ge \epsilon$. Therefore, $0 \le \beta < 1$, and 
							\begin{equation} 
								\label{new:ei}
								\|Tx - Ty\| \le \beta \|x - y\| \ \ \mbox{whenever} \ \epsilon \le \|x - y\| \le 2\epsilon.
							\end{equation}
							We next show $(\forall k \in  \{0, 1, 2, \ldots\})$  
							\begin{equation}
								 \label{new:eii}
								\|Tx - Ty\| \le \beta \|x - y\| \ \ \mbox{whenever} \ 2^k \epsilon \le \|x - y\| \le 2^{k+1} \epsilon.
							\end{equation}
							We proceed by induction.
							Indeed,  \cref{new:eii} holds for $k = 0$ by \cref{new:ei}.
							Now suppose \cref{new:eii}  holds for $k = n$ where $n \ge 0$. 
							Thus suppose $\{x,y \}\subseteq  X$ satisfies that 
							$2^{n+1}\epsilon \le \|x-y\| \le 2^{n+2}\epsilon$.
							Now let $z = (x + y)/2$, then 
							\begin{equation}
							x - z = \tfrac{x-y}{2} = z - y \ \ \mbox{and so}\ \ 2^n \epsilon \le  \|x - z\| \le 2^{n+1} \epsilon, \ 2^n \epsilon \le  \|z - y\| \le 2^{n+1} \epsilon.
							\end{equation}
							Because  $\|x - z\| = \|z - y\| = \frac{1}{2} \|x- y\|$, 
							the triangle inequality yields
							\begin{subequations} 
								\begin{align}
								\| Tx - Ty\| &= \| Tx - Tz + Tz - Ty\| 
								\le \|Tx - Tz\| + \|Tz - Ty\| \\
								&\le \beta\|x - z\| + \beta\|z - y\| = \beta\|x - y\|.
								\end{align}
							\end{subequations}
							It follows by induction that (\ref{new:eii}) 
							is true  $(\forall k \in  \{0, 1, 2, \ldots\})$, and so $T$ is a contraction for large distances.
							
							\cref{new:pi:iii} $\Rightarrow$ \cref{new:pi:i}:
							Clearly  $T$ is a contraction for large distances
							if and only if $-T$ is a contraction for large distances.
							Therefore it is sufficient to show the implication
							[$T$ is a contraction for large distances $\RA$ 
							$T$ is \ssnonex].
							Suppose  $(x_n)_\nnn$, $( y_n)_\nnn$ 
							are sequences in $X$ such that 
							\begin{equation}
								\label{e:lim:cont}
								\|x_n - y_n\|^2 - \|T x_n - Ty_n\|^2 \to 0.
							\end{equation}
							We claim that $(x_n-y_n)\to 0$. Indeed, 
							suppose for eventual contradiction that  $\limsup \|x_n - y_n\| > 0$. 
							After passing to a subsequence and relabelling if necessary
							$(\exists \epsilon > 0)$  
							such that $(\forall \nnn)$ $\|x_{n} - y_{n}\| \ge \epsilon$. 
							This means $(\exists\beta < 1)$ so that
							$\|T x_{n} - T y_{n}\| \le \beta \|x_{n} - y_{n}\|$. Therefore, 
							\begin{equation}
								(\forall \nnn)\quad 	\|x_{n} - y_{n}\|^2 - \|T x_{n} - Ty_{n}\|^2 
								\ge (1-\beta^2) \|x_{n} - y_{n}\|^2 \ge (1 - \beta^2) \epsilon^2  
							\end{equation}
							and this contradicts \cref{e:lim:cont}. 
							Therefore $\|x_n - y_n\| \to 0$ and consequently %by the nonexpansiveness of $T$
							$\|T x_n - Ty_n\| \to 0$ which implies $(x_n - y_n) - (Tx_n - Ty_n) \to 0$.
							Thus $T$ is \ssnonex. 
						\end{proof}
						
											It is clear that every Banach contraction  is a contraction for large distances. 
						However, the opposite is not true as we illustrate 
						in \cref{ex-cont-in-large}  below.
						\begin{example}
							\label{ex-cont-in-large} 
							Let 
							\begin{equation}
								T\colon \RR\to \RR\colon x\mapsto 
								\begin{cases}
									1,&x\ge \tfrac{\pi}{2};
									\\
									\sin x,&\abs{x}<\tfrac{\pi}{2};
									\\
									-1, &\text{otherwise}.
								\end{cases}	
							\end{equation}	  
							Then the following hold:
							\begin{enumerate}
								\item
								\label{ex-cont-in-large:i} 
								$T$ is nonexpansive.
								\item
								\label{ex-cont-in-large:ii}
								$T$is \emph{not} a Banach contraction. 
								\item
								\label{ex-cont-in-large:iii}
								$T$ is a contraction for large distances.
								\item
								\label{ex-cont-in-large:iv}
								Both $T$ and $-T$ are \ssnonex.
								\item
								\label{ex-cont-in-large:v}
								Both $T$ and $-T$ are strongly nonexpansive.
							\end{enumerate}
							\begin{figure}
								\begin{center}
									\includegraphics[scale=0.7]{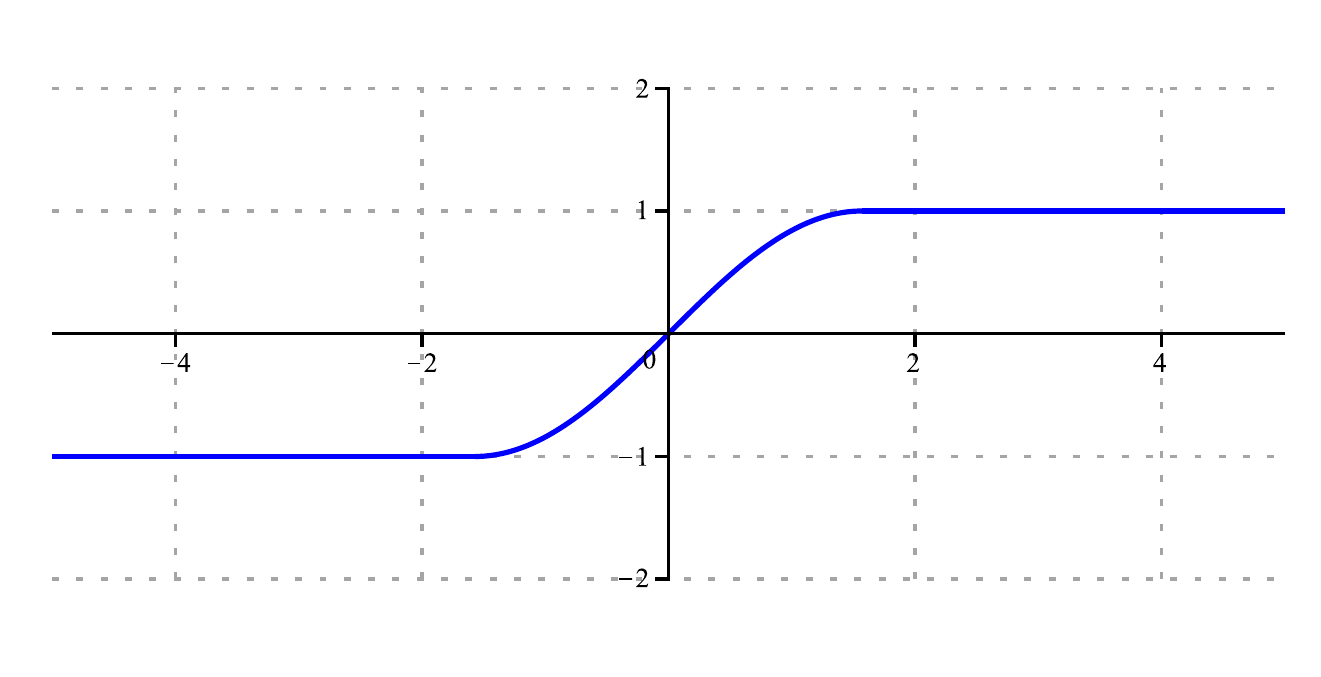}
								\end{center}
								\caption{A \texttt{GeoGebra} snapshot illustrating the mapping
									$T$ in \cref{ex-cont-in-large}. }
							\end{figure}	
						\end{example}
						
						\begin{proof}
							Recall that (see, e.g., \cite[Theorem~5.12]{Beck2017} ) 
							if $T\colon \RR\to \RR$ is differentiable 
							then 
							\begin{equation}
								\label{e:fact}
								\text{
									$T$ is Lipschitz continuous with a constant $K\ge 0$
									if  and only if $(\forall x\in \RR)$ $\abs{T'(x)}\le K$.}
							\end{equation}
							\cref{ex-cont-in-large:i}: 
							One can directly verify that $T$ is differentiable and that
							\begin{equation}
								T'\colon \RR\to \RR\colon x\mapsto 
								\begin{cases}
									\cos x,&\abs{x}<\tfrac{\pi}{2};
									\\
									0, &\text{otherwise}.
								\end{cases}	
							\end{equation}	 
							Hence $(\forall x\in \RR)$ $\abs{T'(x)}\le 1$.
							
							Consequently, by \cref{e:fact},
							$T$ is Lipschitz continuous with a constant  $1$
							and the conclusion follows.
							\cref{ex-cont-in-large:ii}:
							Suppose for eventual contradiction that $T$ is a Banach contraction.
							Then \cref{e:fact} implies that exists $K\in \left[0,1\right[$ such that 
							$(\forall x\in \RR)$ $\abs{T'(x)}\le K<1$. However, $\abs{T'(0)}=\cos 0=1>K$,
							which is absurd.
							
							\cref{ex-cont-in-large:iii}:
							Let $\epsilon > 0$. 
							Observe that if $|t|  \ge  \epsilon/4$ then $|T^\prime(t)| \le \alpha < 1$ 
							where $\alpha  = |T^\prime(\epsilon/4)|$. 
							We choose $\beta = (\alpha + 3)/4$. 
							Now suppose $|x - y| \ge \epsilon$, where $x < y$. 
							In the case $|y| \ge |x|$, we have $y \ge |x -y|/2 \ge \epsilon/2$.  
							Therefore, by the Fundamental theorem of calculus, we write
							
							\begin{equation}
								|Tx - Ty| = \left| \int_x^y T^\prime(t) \, dt \right| \le \int_x^{|x-y|/4} 1 \, dt + \int_{|x-y|/4}^y \alpha \le \beta |x -y|.
							\end{equation}
							
							Similarly, if $|x| \ge |y|$, one has $x \le -|x-y|/2$, 
							and again, $|Tx - Ty| \le \beta |x -y|$. 
							Therefore, $T$ is a contraction for large distances.
							\cref{ex-cont-in-large:iv}--\cref{ex-cont-in-large:v}:
							Combine \cref{ex-cont-in-large:iii}
							with \cref{new:pi}.
						\end{proof}	
						
						The next result and more general variations of it, are 
						well-known in fixed point theory, see, for example, 
						\cite[Corollary, p. 463]{Rak62} and \cite[Theorem 2.1]{AlbGue97}. 
						Nevertheless, we include a simple proof based on  \cref{thm-umg}\cref{thm-umg:ii} for completeness.
						\begin{proposition} 
							\label{cor-cil-fp} 
							Let $T\colon X \to X$ be a contraction for large distances. Let $x_0\in X$
							and set $(\forall \nnn)$ $x_n=T^n x_0$.
							Then $(\exists\bar x\in X)$ such that the following hold:
							\begin{enumerate}
								\item
								\label{cor-cil-fp:i} 
								$\fix T=\{\overline{x}\}$.
								\item 
								\label{cor-cil-fp:ii} 
								$x_n\to \bar x$. 
							\end{enumerate}
						\end{proposition}
						
						\begin{proof} 
							\cref{cor-cil-fp:i}: 
							On the one hand because $T$ is nonexpansive, $T = R_A$ for 
							some maximally monotone operator 
							$A\colon  X \rras  X$ (see \cite[Corollary 23.11, Proposition 4.4]{BC2017}). 
							On the other hand, because $-T=-R_A$  is a contraction 
							 for large distances
							it is \ssnonex\ by  \cref{new:pi}.
							Therefore, $A$ is uniformly monotone by \cref{prop:um:sne}.
							Consequently, 	by 
							\cref{thm-umg}\cref{thm-umg:iii}, $A$ has a unique zero. 
							Now combine this with \cref{eq:fix:JA:RA}.
							
							\cref{cor-cil-fp:ii}: 
							Note that $(\norm{x_n-\overline{x}})_\nnn$
							converges by, e.g., \cite[Proposition~5.4(ii)]{BC2017}.
							Now, suppose by way of contradiction that $(x_n)_\nnn $ 
							does not converge in norm to $\bar x$. 
							Then $\lim_\nnn \|x_n - \bar x\| = \epsilon$ 
							where $\epsilon > 0$. Thus we choose $0 < \beta < 1$ so that 
							\begin{equation}
								(\forall \nnn)\quad \|x_{n+1} - \bar x\| = \|Tx_n - T\bar x\| \le \beta \|x_n -\bar x\|. 
							\end{equation}
							Now for $n$ sufficiently large, we have $\|x_n - \bar x\| < \epsilon/\beta$. Then 
							\begin{equation}
								\|x_{n+1} - \bar x\| \le \beta \|x_n - \bar x\| < \epsilon.
							\end{equation}
							This contradiction completes the proof. 
							Alternatively, use \cref{prop:primal:dual}
							with $(x_n,y_n)_\nnn$ replaced by $(T^n x_0, \overline{x})_\nnn$
							to conclude that $T^n x_0-\overline{x}\to 0$.
						\end{proof}	
						
						\section{Self-Dual Properties on Hilbert Spaces}
						\label{sec:6}
						\begin{lemma} 
							\label{lem-res-a} 
							Suppose $C\colon X \to X$ is uniformly continuous. 
							Then for each $\epsilon > 0$ $(\exists M > 0)$ depending on $\epsilon$ so that
							\begin{equation} 
								\|x - y \| \le M\|u-v\| \ \ \ \mbox{whenever}\ \ \|x - y\| 
								\ge \epsilon, \;\;(u,v)\in J_C x \times J_C y.
							\end{equation}
						\end{lemma}
						
						\begin{proof} Let $\epsilon > 0$. 
							By the uniform continuity of $C$ we choose $0 < \delta < \epsilon/2$ so that
							\begin{equation} \label{eq-eps}
								\|Cu - Cv\| < \frac{\epsilon}{2} \ \ \ \mbox{whenever}\ \ \|u - v\| < \delta.
							\end{equation} 
							Because $C$ is uniformly continuous, it is Lipschitz for large distances
							(see  \cite[Proposition~1.11]{BL2000}). Thus we choose $K > 0$ so that
							\begin{equation} \label{eq-A-lip-large}
								\|Cu - Cv\| \le K \|u - v\| \ \ \ \mbox{whenever}\ \ \|u - v\| \ge \delta.
							\end{equation}
							
							Now let us suppose 
							\begin{equation} \label{eq-lem-res-c}
								\|x - y\| \ge \epsilon, \ \ \ u \in J_C x,\ v \in J_Cy.
							\end{equation}
							We will show that $\|x - y\| \le M\|u - v\|$ where $M = K+1$. First, we verify that $\|u - v\| \ge \delta$ where $\delta$ is from (\ref{eq-eps})  by way of a contradiction. 
							So let us assume to the contrary that $\|u - v\| < \delta$. Then by (\ref{eq-eps}) we have 
							\begin{equation}
								\|Cu - Cv\| < \frac{\epsilon}{2}
							\end{equation}
							Because $u \in J_Cx$ and $v \in J_C y$, this implies $Cu = x - u$ and $Cv = y -v$. Then
							\begin{equation}
								\|x - u - (y - v)\| < \tfrac{\epsilon}{2} \ \Rightarrow \ \|x- y\| - \|u-v\| < \tfrac{\epsilon}{2}\  \Rightarrow \ \|x-y\| < \tfrac{\epsilon}{2} + \delta < \epsilon
							\end{equation}   
							This contradicts (\ref{eq-lem-res-c}), and so $\|u - v\| \ge \delta$. Therefore, using (\ref{eq-A-lip-large}), one has 
							\begin{eqnarray*}
								\|x - y\| &=& \|u + Cu - (v + Cv)\| \le \|u - v\| + \|Cu - Cv\| \\
								&\le& \|u - v\| + K\|u - v\| = M\|u - v\|,
							\end{eqnarray*}
							as desired. 
						\end{proof}
						
						\begin{theorem}
							\label{thm-RA-cont-inl} 
							The following hold:
							\begin{enumerate}
								\item
								\label{thm-RA-cont-inl:i} 
								Suppose $A$ is uniformly monotone and uniformly continuous. 
								Then $R_A$ is a contraction for large distances.
								
								\item
								\label{thm-RA-cont-inl:ii} 
								Suppose $R_A$ is a contraction for large distances. 
								Then $A$ is uniformly monotone with a supercoercive modulus.
							\end{enumerate} 
						\end{theorem}
						
						\begin{proof} 
							\cref{thm-RA-cont-inl:i}:  
							Let $\phi$ be a modulus function for $A$, let 
							$\epsilon > 0$ and suppose $\|x - y\| \ge \epsilon$.    
							On the one hand, it follows from \cref{lem-res-a} that $(\exists K>0)$
							such that
							$\|J_A x - J_A y \| \ge K\|x-y\|\ge K\epsilon $.
							On the other hand, because $\dom A=X$, 
							\cref{prop-uni-mon-mod} implies that 
							$(\exists \alpha > 0)$ such that $(\forall  t \ge K\epsilon)$ 
							$\phi(t) \ge \alpha t^2$.
							Altogether, we learn that 
							\begin{equation}  
								\label{eq-thm-RA-cont-inl-c}
								\phi(\tfrac{1}{2}\|J_A x - J_Ay\|) \ge \tfrac{\alpha K^2}{4}\|x-y\|^2.
							\end{equation}
							Set $\beta=\sqrt{\alpha}K$ and let $(x,y)\in X\times X$.
							Combining \cref{eq-thm-RA-cont-inl-c} and  \cref{prop:um:sn:i} 
							we learn that $\|x - y\| \ge \epsilon\RA
							\|R_A x - R_A y\|^2 + \beta^2\|x - y\|^2 \le \|x - y\|^2 
							$; equivalently,  $ \|x - y\| \ge \epsilon \RA
							\|R_A x - R_A y\| \le(1-\beta) \|x - y\|$.
							That is, $R_A$ is a contraction for large distances. 
							\cref{thm-RA-cont-inl:ii}:
							Let  $\epsilon>0$ and let
							$(x,x^*)\in \gra A$,
							$(y,y^*)\in \gra A$.
							Suppose that $\norm{x-y}\ge \epsilon$.
							Set $(u,v)=(x+x^*, y+y^*)$
							and observe that \cref{eq:gr:A:RA} implies 
							\begin{equation} 
								\label{eq-Minty-rep}
								(x,x^*)=\tfrac{1}{2}(u+R_Au,u-R_Au)
								\text{\;\;and\;\;}
								(y,y^*)=\tfrac{1}{2}(v+R_Av,v-R_Av).
							\end{equation}
							
							It follows from \cref{eq-Minty-rep},
							the nonexpansiveness of $R_A$, 
							and the triangle inequality that 
							\begin{equation} 
								\label{eq-bigger}
								\|x - y\| =\tfrac{1}{2}\norm{u-v-(R_Au-R_Av)}\le \tfrac{1}{2}(\|u - v\| + \|R_A u - R_A v\|) \le\|u - v\|.
							\end{equation}
							Hence $\norm{u-v}\ge \epsilon$.
							Consequently, because $R_A$ is a contraction for large distances, 
							$(\exists \beta \in \left]0,1\right[)$
							such that 
							\begin{equation} 
								\label{eq:opp:i}
								\norm{R_Au-R_Av}\le \beta  \norm{u-v}.
							\end{equation}
							Using \cref{eq-Minty-rep} and \cref{eq:opp:i} 
							we learn that
							\begin{subequations} 
								\begin{align} 
									\langle x - y, x^* - y^*\rangle 
									&= \tfrac{1}{4} \langle u - v + R_A u - R_A v, u - v -( R_Au - R_A v) \rangle \\
									&= \tfrac{1}{4} \left(\|u-v\|^2 - \|R_Au - R_Av\|^2\right) \\
									&\ge \tfrac{1}{4} (1 - \beta^2) \|u - v\|^2   \\
									&\ge \tfrac{1}{4} (1 - \beta^2)\|x - y\|^2.
								\end{align}
							\end{subequations}
							Therefore, for $t \ge \epsilon$, we have 
							\begin{equation}
								\inf\left\{ \langle x - y, x^* - y^* \rangle \, | \, (x,x^*) \in \graph(A), (y,y^*) \in \graph(A), \|x - y\| \ge t\right\} \ge \tfrac{1}{4}(1 - \beta)^2 t^2.
							\end{equation}
							That is, $A$ has a modulus $\phi$ satisfying $\phi(t) \ge \frac{1}{4}(1-\beta^2) t^2$ for $t \ge \epsilon$. 
						\end{proof}
						
						This brings us to our main duality result of this section.
						
						\begin{theorem}
							\label{thm-um-self-dual} 
							The   following are equivalent.
							
							\begin{enumerate}
								\item 
								\label{thm-um-self-dual:i} 
								$A$ is uniformly monotone and uniformly continuous.
								
								\item 
								\label{thm-um-self-dual:ii} 
								$R_A$ is a contraction for large distances.
								
								\item 
								\label{thm-um-self-dual:iii}  Both $A$ and $A^{-1}$ 
								are uniformly monotone with supercoercive moduli.
								
								\item 
								\label{thm-um-self-dual:iv}  
								Both $A$ and $A^{-1}$ are uniformly monotone.
								
								\item 
								\label{thm-um-self-dual:v} 
								$A^{-1}$ is uniformly monotone and uniformly continuous.
								\item 
								\label{thm-um-self-dual:vi} 
								$R_A$ and $R_{A^{-1}}$ are strongly nonexpansive.
							\end{enumerate}
						\end{theorem}
						
						\begin{proof} \cref{thm-um-self-dual:i}  $\Rightarrow$  
							\cref{thm-um-self-dual:ii}: \cref{thm-RA-cont-inl}\cref{thm-RA-cont-inl:i} . 
							
							\cref{thm-um-self-dual:ii}  $\Rightarrow$  
							\cref{thm-um-self-dual:iii}:  Since $R_A$ 
							is a contraction for large distances, so is $R_{A^{-1}} = -R_A$. 
							Thus this follows by applying  
							\cref{thm-RA-cont-inl}\cref{thm-RA-cont-inl:ii}  
							on $R_A$ and on $R_{A^{-1}}$.
							
							\cref{thm-um-self-dual:iii}  $\Rightarrow$  
							\cref{thm-um-self-dual:iv}:  
							is immediate, and \cref{thm-um-self-dual:iv}  $\Rightarrow$  
							\cref{thm-um-self-dual:v}: 
							follows from \cref{thm-umg} applied to $A$ 
							to deduce $A^{-1}$ is uniformly continuous.
							
							\cref{thm-um-self-dual:v}  $\Rightarrow$  
							\cref{thm-um-self-dual:i}: 
							The above implications show \cref{thm-um-self-dual:i} 
							$\Rightarrow$ \cref{thm-um-self-dual:v}, 
							so the reverse implication follows by applying 
							\cref{thm-um-self-dual:i} $\Rightarrow$ 
							\cref{thm-um-self-dual:v} to the operator $A^{-1}$. 
							\cref{thm-um-self-dual:ii} $\siff$ \cref{thm-um-self-dual:vi}: 
							This is a direct consequence of  \cref{new:pi} applied with $T$ replaced by $R_A$.
						\end{proof} 
						
						Let $f\colon X\to \left]-\infty, +\infty\right]$ be convex, lower semicontinuous and proper.
						Recall that (see, e.g., \cite[Corollary~16.30]{BC2017}) 
						\begin{equation}
							\label{eq:f:f^*:subgrad:inv}
							\partial f^*=(\partial f)^{-1},
						\end{equation}	
						that\footnote{Let $f\colon X\to \RR$ be convex.
						Then $f$ is uniformly smooth 
					if $f$ is smooth and $\grad f$ 
				is uniformly continuous.} (see \cite[Theorem~3.5.12]{Za02} )
						\begin{subequations}
							\label{eq:equiv:f:f*:uc:us}
							\begin{align}
								\text{$f$ is uniformly convex}& \text{ $\siff$ 
									%		$f^*$ is differentiable and $\grad f^*$ is uniformly continuous}
									$f^*$ is  uniformly smooth}
								\\& 
								\text{ $\siff $ $\partial f$ is uniformly monotone,}
							\end{align}	
						\end{subequations}	
						and that (see, e.g.,  \cite[Theorem~18.15]{BC2017} )
						\begin{subequations}
							\label{eq:equiv:f:f*:sc:ls}
							\begin{align}
								\text{$f$ is strongly convex}& \text{ $\siff$ 
									$f^*$ is differentiable and $\grad f^*$ is Lipschitz continuous}
								\\& 
								\text{ $\siff $ $\partial f$ is  monotone}.
							\end{align}	
						\end{subequations}	
						\begin{example}
							\label{ex-f-f*-uc}
							Let 
							\begin{equation}
								f\colon \RR\to \RR\colon x\mapsto 
								\begin{cases}
									4x^2-2,&x\le -1;
									\\
									2x^4,&-1<x<0;
									\\
									x^{3/2},&0\le x<1;
									\\
									\tfrac{3}{4}x^2+\tfrac{1}{4}, &\text{otherwise}.
								\end{cases}	
							\end{equation}
							Then $f$ is differentiable and $f^\prime$ is continuous 
							and increasing.
							Moreover the following hold:
							\begin{enumerate}
								\item
								\label{ex-f-f*-uc:i}
								Both $f$ and $f^*$ are uniformly convex.
								\item
								\label{ex-f-f*-uc:ii}
								Both $f$ and $f^*$ are uniformly smooth.
								\item
								\label{ex-f-f*-uc:iii}
								Both $f^{\prime}$ and $(f^*)^{\prime}$ are 
								uniformly monotone.
								\item
								\label{ex-f-f*-uc:iv}
								Both $f^{\prime}$ and $(f^*)^{\prime}$ are   uniformly continuous.
								\item
								\label{ex-f-f*-uc:v}
								Neither $f$ nor $f^*$ is  strongly convex.	
							\end{enumerate}
						\end{example}
						\begin{proof}
							It is straightforward to verify that 	
							$f$ is differentiable and that
							\begin{equation}
								\label{eq:def:f'}
								f^\prime\colon
								\RR\to \RR\colon x\mapsto 
								\begin{cases}
									8x,&x\le -1;
									\\
									8x^3,&-1<x<0;
									\\
									\tfrac{3}{2}x^{1/2},&0\le x<1;
									\\
									\tfrac{3}{2}x, &\text{otherwise}.
								\end{cases}	
							\end{equation}
							Hence, $f^\prime$	is continuous and strictly 
							increasing as claimed.
							\cref{ex-f-f*-uc:i}:
							It follows from 
							\cite[Theorem~3.1]{Za83}
							that $f$ is uniformly convex
							on bounded sets.
							It follows from the strong (hence uniform)
							convexity of $x^2$ that
							$f$ is uniformly convex.
							We now prove the uniform convexity of $f^*$.
							In view of 
							\cref{eq:equiv:f:f*:uc:us}
							it suffices to verify that $f$
							is uniformly smooth which can be easily deduced from 
							\cref{eq:def:f'}.
							
							\cref{ex-f-f*-uc:ii}\& \cref{ex-f-f*-uc:iii}:
							Combine \cref{ex-f-f*-uc:i}
							and \cref{eq:equiv:f:f*:uc:us}.
							\cref{ex-f-f*-uc:iv}:
							This follows from \cref{ex-f-f*-uc:ii}.
							
							\cref{ex-f-f*-uc:v}:
							On the one hand $2x^4$ is \emph{not}
							strongly convex, hence $f$ is \emph{not}
							strongly convex.
							On the other hand, 
							$(x^{{3}/{2}})^\prime=\tfrac{3}{2}\sqrt{x}$
							is \emph{not} Lipschitz continuous on $\opint{0,1}$.
							Therefore, in view of \cref{eq:equiv:f:f*:sc:ls}
							applied with $f$ replaced by $f^*$
							we conclude that $f^*$ is \emph{not}
							strongly convex.
						\end{proof}		
						
						\begin{example}
							\label{ex-cont-in-large-b} 	
							Let 
							\begin{equation}
								T\colon \RR\to \RR\colon x\mapsto 
								\begin{cases}
									1,&x\ge \tfrac{\pi}{2};
									\\
									\sin x,&\abs{x}<\tfrac{\pi}{2};
									\\
									-1, &\text{otherwise}.
								\end{cases}	
							\end{equation}
							Set $A=\Big(\tfrac{\Id-T}{2}\Big)^{-1}-\Id$. Then 
							there exists a proper lower semicontinuous convex function 
							$	f\colon \RR\to \ocint{-\infty,+\infty}$
							such that 
							\begin{equation}
								\label{ex-cont-in-large-b:ii:b}
								A=f^\prime. 
							\end{equation}
							Moreover, the following hold:
							
							\begin{enumerate}
								
								\item
								\label{ex-cont-in-large-b:iii}
								Both $A$ and $A^{-1}$ are  maximally monotone and uniformly monotone  with supercoercive modulus.
								\item
								\label{ex-cont-in-large-b:iii:i}
								Both $A$ and $A^{-1}$ are   uniformly continuous.
								\item
								\label{ex-cont-in-large-b:iv}
								Both $A$ and $A^{-1}$ are  \emph{not} strongly monotone.
								\item
								\label{ex-cont-in-large-b:v}
								Both $f$ and $f^*$ are uniformly convex  and uniformly smooth.
								\item
								\label{ex-cont-in-large-b:vi}
								Both $f$ and $f^*$ are  \emph{not} strongly convex.	
							\end{enumerate}

						\end{example}
						
						\begin{proof}
							It is clear that $T=-R_A$. 
							Because $T$ is nonexpansive, from 
							\cref{ex-cont-in-large}\cref{ex-cont-in-large:i}
							we learn that $A$ is maximally monotone.
							Therefore, by e.g., \cite[Corollary~22.23]{BC2017}
							there exists a proper lower semicontinuous convex function 
							$	f\colon \RR\to \ocint{-\infty,+\infty}$
							such that
							$A=\partial f$. 
							Finally, we show in \cref{ex-cont-in-large-b:v} below
							that $f$ is uniformly smooth, hence $A=f^\prime$.
							
							\cref{ex-cont-in-large-b:iii}\&\cref{ex-cont-in-large-b:iii:i}: 
							Combine \cref{ex-cont-in-large}\cref{ex-cont-in-large:iii}, 
							%		\cref{ex-cont-in-large-b:i}
							and \cref{thm-um-self-dual} .
							
							\cref{ex-cont-in-large-b:iv}: 
							Suppose for eventual contradiction 
							that $A$ is strongly monotone. Then by \cite[Proposition~4.3(ii)]{BMW2021}   $(\exists \alpha\in\left]0,1\right[)$ such that 
							$-R_A$ is $\alpha$-averaged.
							It follows from \cite[Proposition~4.35]{BC2017} that $(\forall (x,y)\in \RR\times \RR)$
							$\tfrac{1-\alpha}{\alpha}({(\Id+T)x-(\Id+T)y})^2\le ({x-y})^2-({Tx-Ty})^2$. In particular,
							for $y=0$ the above inequality yields 
							$(\forall x\in \left]0,\tfrac{\pi}{2}\right[)$
							\begin{equation}
								\frac{1-\alpha}{\alpha}(x+\sin x)^2\le x^2-\sin^2x.
							\end{equation}	
							Simplifying and multiplying both sides of the above inequality by
							$\alpha$ we obtain 
							$(1-2\alpha)x^2+2(1-\alpha)x\sin x+\sin^2 x\le 0$.
							Rearranging yields $(x+\sin x)^2\le 2\alpha x(x+\sin x)$.
							Because $x\in \left]0,\tfrac{\pi}{2}\right[$, it follows that
							$x+\sin x>0$ and therefore the last inequality is equivalent to
							$1+\tfrac{\sin x}{x}\le 2\alpha$. Taking the limit as $x\to 0^+$
							we learn that  $2\leftarrow 1+\tfrac{\sin x}{x}\le 2\alpha$ which is absurd.
							Hence, $-R_A$ is not averaged; equivalently,  $A$
							is not strongly monotone as claimed.
							
							Using similar argument, one can show that $R_A=-R_{A^{-1}}$ is not averaged; equivalently,  $A^{-1}$
							is not strongly nonexpansive as claimed.
							
							\cref{ex-cont-in-large-b:v}: Combine \cref{ex-cont-in-large-b:iii:i}, \cref{ex-cont-in-large-b:ii:b}
							and \cite[Theorem~3.5.10]{Za83}
							in view of \cref{eq:f:f^*:subgrad:inv}.
							
							\cref{ex-cont-in-large-b:vi}: 
							Combine \cref{ex-cont-in-large-b:iv}, 	
							\cref{ex-cont-in-large-b:ii:b}
							and \cite[Example~22.4(iv)]{BC2017}.
						\end{proof}	
						\begin{remark}
							In   \cref{App:B} we provide finer conclusions about
							the operator $A$ and the function $f$ introduced in 
							\cref{ex-cont-in-large-b}. 
						\end{remark}
						
						\begin{example}
							\label{ex:only:primal}
							Let $f\colon \RR \to \RR\colon x\mapsto \tfrac{1}{4}x^4$ and  let $A = f^\prime$.
							Let $a\colon \RR \to \RR\colon x\mapsto \sqrt[3]{108 x+12\sqrt{81x^2+12}}$ and set 
							\begin{equation}
								T\colon \RR \to \RR\colon x\mapsto x-2\Big(\tfrac{a(x)}{6}-\tfrac{2}{a(x)}\Big).
							\end{equation}	
							Then the following hold:
							\begin{enumerate}
								\item
								\label{ex:only:primal:i}
								$T=-R_A$.
								\item
								\label{ex:only:primal:ii}
								$A=x^3$ is maximally monotone but \emph{not} uniformly continuous. 
								\item
								\label{ex:only:primal:iii}
								$A^{-1}=\sqrt[3]{x}$ is maximally monotone and uniformly continuous. 
								\item
								\label{ex:only:primal:ii:i}
								$A$ is uniformly monotone. 
								\item
								\label{ex:only:primal:iii:i}
								$A^{-1}$ is  \emph{not}  uniformly monotone. 
								\item 
								\label{ex:only:primal:iv}
								$f$ is  uniformly convex. 
								\item 
								\label{ex:only:primal:v}
								$f^*$ is  \emph{not} uniformly convex. 
								\item
								\label{ex:only:primal:vi}
								$T$ is   \ssnonex\ and  strongly nonexpansive.
								\item
								\label{ex:only:primal:vii}
								$-T$ is \emph{neither}  \ssnonex\ \emph{nor} strongly nonexpansive.
								
							\end{enumerate}	
						\end{example}
						
						\begin{proof}
							\cref{ex:only:primal:i}:	
							It is enough to show that $(\forall x\in \RR)$
							$J_A(x)=\tfrac{a(x)}{6}-\tfrac{2}{a(x)}$, equivalently, to show that
							$(\Id+A)\big(\tfrac{a(x)}{6}-\tfrac{2}{a(x)}\big)=x$ .
							Indeed,
							\begin{subequations}
								\begin{align}
									(\Id+A)\big(\tfrac{a(x)}{6}-\tfrac{2}{a(x)}\big)
									&=	\tfrac{a(x)}{6}-\tfrac{2}{a(x)}+\tfrac{(a(x))^3}{6^3}-\tfrac{a(x)}{6}+\tfrac{2}{a(x)}-\tfrac{2^3}{(a(x))^3}
									\\
									&=\tfrac{(a(x))^3}{6^3}-\tfrac{2^3}{(a(x))^3}=\tfrac{(a(x))^6-12^3}{6^3(a(x))^3}=x.
								\end{align}	
							\end{subequations}
							
							\cref{ex:only:primal:ii}\&\cref{ex:only:primal:iii}:	
							This is clear. 
							
							\cref{ex:only:primal:ii:i}--\cref{ex:only:primal:v}:	
							Combine \cref{ex:only:primal:ii}\&\cref{ex:only:primal:iii}
							with \cite[Theorem~3.5.10]{Za83}.

							\cref{ex:only:primal:vi}:
							Combine \cref{ex:only:primal:i}, \cref{ex:only:primal:ii:i}
							and \cref{cor-real-line-eq}.
							
							\cref{ex:only:primal:vii}:		
							Combine \cref{ex:only:primal:i}, \cref{ex:only:primal:iii:i}
							and \cref{cor-real-line-eq}.
						\end{proof}

						\section{Compositions}
						\label{sec:7}
						In this section we examine the behaviour of
						strongly nonexpansive mappings, \ssnonex\ mappings and 
						contractions for large distances under structured compositions.
						The proof of the next result
						follows along the lines of the proof of \cite[Proposition~1.1]{BR77}.		
						\begin{proposition}
							\label{prop:comp:sne}	
							Let $T_1\colon X\to X$ and $T_2\colon X\to X$ be nonexpansive.
							Set $T=T_2T_1$. Then the following hold:
							\begin{enumerate}
								\item
								\label{prop:comp:0}
								Suppose that $T_1$ is strongly nonexpansive and $T_2$ is strongly nonexpansive.
								Then $T $ is strongly nonexpansive.
								\item
								\label{prop:comp:i}
								Suppose that $-T_1$ is strongly nonexpansive and $-T_2$ is strongly nonexpansive.
								Then $T $ is strongly nonexpansive.
								\item
								\label{prop:comp:ii}
								Suppose that $-T_1$ is strongly nonexpansive and $T_2$ is strongly nonexpansive.
								Then $-T $ is strongly nonexpansive.
								\item
								\label{prop:comp:iii}
								Suppose that $T_1$ is strongly nonexpansive and $-T_2$ is strongly nonexpansive.
								Then $-T $ is strongly nonexpansive.
							\end{enumerate}
						\end{proposition}	
						\begin{proof}
							\cref{prop:comp:0}:
							This is 	\cite[Proposition~1.1]{BR77}.
							\cref{prop:comp:i}:
							Clearly $T$ is nonexpansive.
							Now suppose that	$(x_n-y_n)_\nnn$ is bounded
							and that  
							$\norm{x_n-y_n}-\norm{Tx_n-Ty_n}\to 0$.
							Observe that the nonexpansiveness of $T_1$, $T_2$ and, consequently, $T$ implies that  
							\begin{equation}
								0\le \norm{x_n-y_n}-\norm{Tx_n-Ty_n}=\underbrace{\norm{x_n-y_n}
									-\norm{T_1x_n-T_1y_n}}_{\ge 0}
								+\underbrace{\norm{T_1x_n-T_1y_n}-\norm{Tx_n-Ty_n}}_{\ge 0}\to 0.
							\end{equation}	
							Consequently, we learn that
							\begin{subequations}
								\begin{align}
									\norm{x_n-y_n}-\norm{(-T_1)x_n-(-T_1)y_n}=\norm{x_n-y_n}-\norm{T_1x_n-T_1y_n}\to 0,\\
									\norm{T_1x_n-T_1y_n}-\norm{(-T_2)(T_1x_n)-(-T_2)(T_1y_n)}=\norm{T_1x_n-T_1y_n}-\norm{Tx_n-Ty_n}\to 0
								\end{align}	
							\end{subequations}
							The nonexpansiveness  of $T_1$ implies that 
							$\norm{T_1x_n-T_1y_n}$ is bounded.	
							Recalling that $-T_1$ is strongly nonexpansive and $-T_2$ is strongly nonexpansive
							we obtain
							\begin{subequations}
								\begin{align}
									(x_n-y_n)+(T_1x_n-T_1y_n)&\to 0
									\\
									(T_1x_n-T_1y_n)+(T_2T_1x_n-T_2T_1y_n)&\to 0.
								\end{align}
							\end{subequations}
							Hence, 
							\begin{equation}
								(x_n-y_n)-(Tx_n-Ty_n)=(x_n-y_n)+(T_1x_n-T_1y_n)-((T_1x_n-T_1y_n)+(T_2T_1x_n-T_2T_1y_n))\to 0.
							\end{equation}
							\cref{prop:comp:ii}:
							Proceeding similar to 	\cref{prop:comp:i} we learn that 	
							\begin{subequations}
								\begin{align}
									(x_n-y_n)+(T_1x_n-T_1y_n)&\to 0
									\\
									(T_1x_n-T_1y_n)-(T_2T_1x_n-T_2T_1y_n)&\to 0.
								\end{align}
							\end{subequations}
							Hence, 
							\begin{equation}
								(x_n-y_n)+(Tx_n-Ty_n)
								=(x_n-y_n)+(T_1x_n-T_1y_n)-((T_1x_n-T_1y_n)-(T_2T_1x_n-T_2T_1y_n))\to 0.
							\end{equation}
							\cref{prop:comp:iii}:
							Proceed similar to 	\cref{prop:comp:ii}. 
						\end{proof}
						
						The following analogous result 
						holds for \vsne\ mappings.
						\begin{proposition}
							\label{prop:comp:vne}	
							Let $T_1\colon X\to X$ and $T_2\colon X\to X$ be nonexpansive.
							Set $T=T_2T_1$. Then the following hold:
							\begin{enumerate}
								\item
								\label{prop:comp:vne:0}
								Suppose that $T_1$ is \vsne\ and $T_2$ is \vsne.
								Then $T $ is \vsne.
								\item
								\label{prop:comp:vne:i}
								Suppose that $-T_1$ is \vsne\ and $-T_2$ is \vsne.
								Then $T $ is \vsne.
								\item
								\label{prop:comp:vne:ii}
								Suppose that $-T_1$ is \vsne\ and $T_2$ is \vsne.
								Then $-T $ is \vsne.
								\item
								\label{prop:comp:vne:iii}
								Suppose that $T_1$ is \vsne\ 
								and $-T_2$ is \vsne.
								Then $-T $ is \vsne.
							\end{enumerate}
						\end{proposition}	
						\begin{proof}
							\cref{prop:comp:vne:0}:
							Let $(x,y)\in X\times X$
							and suppose that 
							$(x_n)_\nnn$ and $(y_n)_\nnn$
							are sequences in $X$ such that
							$0\le \norm{x_n-y_n}^2-\norm{Tx_n-Ty_n}^2\to 0$.
							Rewrite the above limit as
							\begin{equation}
								0\le \norm{x_n-y_n}^2-\norm{T_1x_n-T_1y_n}^2
								+\norm{T_1x_n-T_1y_n}^2-\norm{Tx_n-Ty_n}^2\to 0,
							\end{equation}	
							and observe that the nonexpansiveness 	of $T_1$
							and $T_2$ implies 
							\begin{equation}
								\norm{x_n-y_n}^2-\norm{T_1x_n-T_1y_n}^2\to 0
								\quad\text{and}\quad 
								\norm{T_1x_n-T_1y_n}^2-\norm{Tx_n-Ty_n}^2\to 0.
							\end{equation}	 
							Because  $T_1$
							and $T_2$ are \vsne\ we learn that 
							\begin{subequations}
								\begin{align}
									(x_n-y_n)-(T_1x_n-T_1y_n)&\to 0
									\label{eq:vsne:comp:i}
									\\	
									(T_1x_n-T_1y_n)-(Tx_n-Ty_n)&\to 0.
									\label{eq:vsne:comp:ii}
								\end{align}		
							\end{subequations}	
							Adding 	\cref{eq:vsne:comp:i} and \cref{eq:vsne:comp:ii}
							yields $(x_n-y_n)-(Tx_n-Ty_n)\to 0$, hence $T$
							is \vsne\ as claimed.
							\cref{prop:comp:vne:i}--\cref{prop:comp:vne:iii}:
							Proceed similar to 	the proof of 
							\cref{prop:comp:sne}\cref{prop:comp:i}--\cref{prop:comp:iii}. 
						\end{proof}
						
						We now turn to compositions of finitely many mappings each of which is either
						strongly nonexpansive or its negative is 
						strongly nonexpansive.
						\begin{theorem}
							\label{thm:comp:m:sne}
							Let $m\ge 2$, let $I=\{1,\ldots , m\}$,
							let $J\subseteq I$, and let $(T_i)_{i\in I}$
							be a family of nonexpansive mappings from 
							$X$ to $X$.
							Suppose that $(\forall i\in I\smallsetminus J)$
							$T_i$ is strongly nonexpansive
							and that 
							$(\forall j\in J)$
							$-T_j$ is strongly nonexpansive.
							Set 
							\begin{equation}
								T=T_m\ldots T_1. 
							\end{equation}		
							Then $(-1)^{|J|}T$ is strongly nonexpansive.
						\end{theorem}	
						\begin{proof}
							We proceed by induction on $k\in \{2,\ldots,m\}$.
							To this end, let us set $(\forall k\in \{2,\ldots,m\})$
							$J_k=\menge{j\in J}{j\le k}$.	
							By \cref{prop:comp:sne} the claim is true for 
							$k=2$.
							Now assume that,
							for some $k\in \{2,\ldots,m\}$,
							we have $(-1)^{|J_k|}T_k\ldots T_1$
							is strongly nonexpansive.
							If $T_{k+1}$ is strongly nonexpansive
							then $|J_{k+1}|=|J_k|$
							and the conclusion follows by applying 
							\cref{prop:comp:sne}\cref{prop:comp:0}
							(respectively \cref{prop:comp:sne}\cref{prop:comp:ii})
							with $(T_1,T_2)$ replaced by $(T_m\ldots T_1,T_{m+1})$ in the case $|J_k|$ is even
							(respectively $|J_k|$ is odd).
							If $-T_{k+1}$ is strongly nonexpansive
							then $|J_{k+1}|=|J_k|+1$
							and the conclusion follows by applying 
							\cref{prop:comp:sne}\cref{prop:comp:i}
							(respectively \cref{prop:comp:sne}\cref{prop:comp:iii})
							with $(T_1,T_2)$ replaced by $(T_m\ldots T_1,T_{m+1})$ in the case $|J_k|$ is even
							(respectively $|J_k|$ is odd).
						\end{proof}	
						
						\begin{theorem}
							\label{thm:comp:m:vsne}
							Let $m\ge 2$, let $I=\{1,\ldots , m\}$,
							let $J\subseteq I$, and let $(T_i)_{i\in I}$
							be a family of nonexpansive mappings from 
							$X$ to $X$.
							Suppose that $(\forall i\in I\smallsetminus J)$
							$T_i$ is \vsne\
							and that 
							$(\forall j\in J)$
							$-T_j$ is \vsne.
							Set 
							\begin{equation}
								T=T_m\ldots T_1. 
							\end{equation}		
							Then $(-1)^{|J|}T$ is \vsne.
						\end{theorem}	
						\begin{proof}
							Proceed similar to the proof of
							\cref{thm:comp:m:sne} 
							but use 
							\cref{prop:comp:vne}\cref{prop:comp:vne:0}--\cref{prop:comp:vne:iii} instead of 
							\cref{prop:comp:sne}\cref{prop:comp:0}--\cref{prop:comp:iii}.
						\end{proof}	
						
						We conclude this section with the following result 
						concerning compositions that involve contractions for large distances.
						\begin{proposition}
							\label{lem:comp:cont:large}
							Let $m\ge 2$, let $I=\{1,\ldots , m\}$,
							let $\overline{m}\in I$, and let $(T_i)_{i\in I}$
							be a family of nonexpansive mappings from 
							$X$ to $X$ and suppose that $T_{\overline{m}}$ 
							is a contraction for large distances.
							Set 
							\begin{equation}
								T=T_m\ldots T_1. 
							\end{equation}		
							Then $T$ is a contraction for large distances.
						\end{proposition}	
						\begin{proof}
							Let $(x,y)\in X\times X$
							and let $\epsilon >0$. Suppose that 
							$\norm{x-y}\ge \epsilon$.
							We proceed by induction on $k\in \{2,\ldots,m\}$.
							For $m=2$ we examine two cases:
							\textsc{Case~1:} $T_1$ is a contraction for large distances.
							Then $(\exists \beta_\epsilon\in \left]0,1\right[)$
							such that $\norm{T_1x-T_1y}\le \beta_\epsilon \norm{x-y}$.
							Therefore, because $T_2 $ is nonexpansive 
							we learn that
							$\norm{T x-T y}\le \norm{T_1x-T_1y}\le \beta_\epsilon \norm{x-y}$.
							\textsc{Case~2:} $T_2$ is a contraction for large distances.
							If $\norm{T_1x-T_1y}\ge \tfrac{\epsilon}{2}$.
							Then$(\exists \alpha_\epsilon\in \left]0,1\right[)$
							such that $\norm{Tx-Ty}\le \alpha_{\epsilon/2} \norm{x-y}$.
							Now suppose that $\norm{T_1x-T_1y}< \tfrac{\epsilon}{2}$.
							Then $\norm{Tx-Ty}\le\norm{T_1x-T_1y}< \tfrac{\epsilon}{2}\le \tfrac{1}{2}\norm{x-y} $.
							Setting $\beta_\epsilon=\max\{\alpha_{\epsilon/2},\tfrac{1}{2} \}$ 
							proves the claim for $m=2$.
							Now assume that,
							for some $k\in \{2,\ldots,m\}$,
							we have $T_k\ldots T_1$ is a contraction for large distances
							whenever $T_{\overline{k}}$ is a contraction for large distances,
							$\overline{k}\in \{1,\ldots,k\}$.
							Consider the composition $T_{k+1}T_k\ldots T_1$.
							If $(\exists \overline{k}\in \{1,\ldots, k\})$
							such that $T_{\overline{k}}$ is a contraction 
							for large distances
							then the inductive hypothesis implies that
							$T_k\ldots T_1$ is a contraction for large distances.
							Otherwise, $T_{k+1} $ must be a contraction for large distances.
							In both cases
							the conclusion follows from applying the base 
							case with $(T_1,T_2)$ replaced by $(T_k\ldots T_1,T_{k+1})$.
							The proof is complete.
						\end{proof}	
						
						%	\section{Application to Peaceman--Rachford  algorithm}
						
						\section{Application to splitting  algorithms}	
						\label{sec:8}
						In this section we use  our earlier conclusions
						to obtain stronger and more refined convergence results
						for some important splitting methods (see, e.g., \cite[Chapter~26]{BC2017}); 
						namely, Peaceman--Rachford algorithm
					(see \cref{thm:PR}\cref{thm:PR:iii:a}--\cref{thm:PR:iii:b}\&\cref{thm:PR:ii} below),
					Douglas--Rachford algorithm
					(see \cref{thm:DR}\cref{thm:DR:ii}  below)
					and forward-backward algorithm
					(see \cref{thm:FB}\cref{thm:FB:ii}  below).

						Let $C\colon X\rras X$ be uniformly monotone
						and suppose that $\zer C\neq \fady$.
						Then $C$ is strictly monotone and it follows from,
						e.g., \cite[Proposition~23.35]{BC2017} that 
						\begin{equation}
							\label{eq:zer:uniq}
							\text{$\zer C$ is a singleton}.
						\end{equation}

						\begin{theorem}[{\bf Peaceman--Rachford algorithm}]
							\label{thm:PR}
							%		Let $A\colon X\rras X$ and $B\colon X\rras X$ be maximally monotone 
							%		operators such $\zer(A+B)\neq \fady$.
							Suppose that $A$ is uniformly monotone. Set $T=R_BR_A$.
							Let $x_0\in X$ and set $(\forall \nnn)$:
							\begin{subequations}
								\begin{align}
									x_{n+1}&=Tx_n,
									\\
									y_{n}&=J_A x_n.
								\end{align}
							\end{subequations}	
							Then the following hold.
							\begin{enumerate}
								
								\item
								Suppose that  $\zer(A+B)\neq \fady$. Then we have:
								\label{thm:PR:i}
								\begin{enumerate}	
									\item 
									\label{thm:PR:i:a}
									$\zer(A+B)$ is a singleton and  $\fix T\neq \fady$.  
									\item 
									\label{thm:PR:i:b}
									$(\exists \overline{x}\in \fix T)$
									$(y_n)_\nnn$ converges strongly to $J_A\overline{x}$
									and $\zer(A+B)=\{J_A\overline{x}\}$.		
								\end{enumerate}	
								If, in addition, $B$ is uniformly monotone then we also have: 
								\begin{enumerate}	
									\setcounter{enumii}{2}
									\item 
									\label{thm:PR:iii:a}
									$T$ is strongly nonexpansive.
									\item 
									\label{thm:PR:iii:b}
									$(x_n)_\nnn$ converges weakly to $\overline{x}$.
								\end{enumerate}
								
								\item	
								\label{thm:PR:ii}
								Suppose  that 
								$A^{-1}$ is uniformly monotone. Then we  have: 
								\begin{enumerate}	
									\item 
									\label{thm:PR:ii:a}
									$T$ is a contraction for large distances and 
									$(\exists \overline{x}\in X)$ such that
									$\fix T=\{\overline{x}\}$.
									
									\item 
									\label{thm:PR:ii:b}
									$(x_n)_\nnn$ converges strongly to $\overline{x}$.
								\end{enumerate}	
							\end{enumerate}	
						\end{theorem}	
						
						\begin{proof}
							\cref{thm:PR:i:a}:
							Applying \cref{eq:zer:uniq} with $C$ replaced by $A+B$, we conclude that  
							$\zer(A+B) $ is a singleton.
							It follows from \cite[Proposition~26.1(iii)(b)]{BC2017} that $\fix T\neq \fady$. 
							\cref{thm:PR:i:b}:
							This is \cite[Proposition~26.13]{BC2017}.

							\cref{thm:PR:iii:a}: It follows from \cref{prop:um:sne} applied to $A$
							(respectively $B$)  that $-R_A$ and (respectively $-R_B$) is strongly nonexpansive.
							Consequently, $R_BR_A$ is strongly nonexpansive by 
							\cref{prop:comp:sne}\cref{prop:comp:i} 
							applied with $(T_1,T_2)$ replaced by $(R_A,R_B)$.	
							
							\cref{thm:PR:iii:b}: Combine 	\cref{thm:PR:iii:a} with \cref{fact:T:sne:converges}.
							
							\cref{thm:PR:ii:a}\&\cref{thm:PR:ii:b}: It follows from \cref{thm-um-self-dual} 
							that $R_A$ is a contraction for large distances. Combining this with 
							\cref{lem:comp:cont:large} applied with $(m,T_1,T_2) $ replaced 
							by $(2, R_A,R_B)$ we learn that $T$ is a contraction for large distances.
							Now combine this with \cref{cor-cil-fp}\cref{cor-cil-fp:i}\&\cref{cor-cil-fp:ii}.
						\end{proof}	
						
						The assumption that $B$ is uniformly monotone is critical in the conclusion of
						\cref{thm:PR}\cref{thm:PR:iii:b} as we illustrate below.
						\begin{example}
							\label{ex:no:conv}
							Suppose that $X\neq \{0\}$.
							Let $A=N_{\{0\}}$ and let $B\equiv 0$.
							Then $A$ is strongly monotone, hence uniformly monotone, and $B$ is \emph{not}
							uniformly monotone.
							Moreover, $R_A=-\Id $, $R_B=\Id$. Consequently $T= R_BR_A=-\Id$.
							Let $x_0\in X\smallsetminus \{0\}$.
							Then $(\forall \nnn)$ $T^n x_0=(-1)^nx_0$ and $(T^n x_0)_\nnn$
							does \emph{not} converge.
						\end{example}
						The assumption that $A^{-1}$ is uniformly monotone is critical in the conclusion of
						\cref{thm:PR}\cref{thm:PR:ii:b} as we illustrate below.
						\begin{example}
							\label{ex:weak-not-strong:conv}
							Suppose that $X=\ell^2(\{1,2,3,\dots\})$
							with the standard Schauder basis $e_1=(1,0,\ldots)$,
							$e_2=(0,1,0,\ldots)$, and so on.
							Let $A=N_{\{0\}}$, let 
							$R\colon  X\to X\colon (x_1,x_2,\ldots)\mapsto (0,x_1,x_2,\ldots)$
							(the right shift operator)
							and set $B=\tfrac{1}{2}(\Id-R)^{-1}-\Id$.
							Then $A$ is strongly monotone, hence uniformly monotone.
							Moreover, because $R_B=-R$ is nonexpansive, we conclude that
							$B$
							is maximally monotone by \cref{fact:corres}.  
							Observe that $A^{-1}\equiv 0$ is \emph{not} uniformly monotone.
							Moreover, $R_A=-\Id $, $R_B=-R$.
							Consequently $T= R_BR_A=R$.
							Let $x_0=e_1$.
							Then $(\forall \nnn)$ $T^n x_0=e_{n+1}$
							and $(T^n x_0)_\nnn=(e_1,e_2,\ldots)$
							converges weakly but \emph{not} strongly to $0$.
						\end{example}
						
						We now turn to Douglas--Rachford algorithm.
						We recall the following fact.
						\begin{fact}
							\label{fact:strong:comb}
							Let $T_1\colon X\to X$, let $T_2\colon X\to X$  and let $\lambda\in \left]0,1\right[$.
							Set $T=(1-\lambda) T_1+\lambda T_2$. Suppose that $T_1$ is strongly nonexpansive and 
							that $T_2 $ is nonexpansive. Then $T$  is strongly nonexpansive.
						\end{fact}	
						\begin{proof}
							See \cite[Proposition~1.3]{BR77}.
						\end{proof}	
						
						\begin{proposition}
							\label{prop:DR:cont}
							Let $R\colon X\to X$ and let $\lambda\in \left]0,1\right[$.
							Set $T=(1-\lambda) \Id+\lambda R$.
							Suppose that $-R$ is strongly nonexpansive.
							Then $T$ is a contraction for large distances.
						\end{proposition}	
						
						\begin{proof}
							Clearly $R$ is nonexpansive.
							Observe that \cref{fact:strong:comb} applied with $(T_1,T_2)$
							replaced by $(\Id, R)$ implies that $T$ is strongly nonexpansive.
							We claim that $-T$ is   strongly nonexpansive.
							Indeed, applying \cref{fact:strong:comb} with $(T_1,T_2,\lambda)$
							replaced by $( -R,-\Id,1-\lambda)$ implies that $-T=(1-\lambda)(-R)+\lambda (-\Id)$
							is   strongly nonexpansive.
							Altogether, we conclude that $T$ is a contraction for large distances in view of
							\cref{new:pi}.
						\end{proof}
						
						\begin{theorem}[{\bf Douglas--Rachford algorithm}]
							\label{thm:DR}
							Suppose that $\zer(A+B)\neq \fady$.
							Suppose that $A$ is uniformly monotone. 
							Set $T=\tfrac{1}{2}(\Id+R_BR_A)$.
							Let $x_0\in X$ and set $(\forall \nnn)$:
							\begin{subequations}
								\begin{align}
									x_{n+1}&=Tx_n,
									\\
									y_{n}&=J_A x_n.
								\end{align}
							\end{subequations}	
							Then $\zer(A+B)$ is a singleton, $\fix T\neq \fady$ and $(\exists \overline{x}\in \fix T)$ 
							such that the following hold:
							\begin{enumerate}
								
								\item
								\label{thm:DR:i}
								$(y_n)_\nnn$ converges strongly to 
								$\overline{y}\coloneqq J_A\overline{x}$.			
								\item	
								\label{thm:DR:ii}
								Suppose  that $(\exists C\in \{A^{-1},B^{-1}\})$
								such that $C$ is uniformly monotone.
								
								Then we additionally have: 
								\begin{enumerate}	
									\item 
									\label{thm:DR:iii:a}
									$T$ is a contraction for large distances  and 
									$(\exists \overline{x}\in X)$ such that
									$\fix T=\{\overline{x}\}$.
									\item 
									\label{thm:DR:iii:b}
									$(x_n)_\nnn$ converges strongly to $\overline{x}$.
								\end{enumerate}

							\end{enumerate}	
						\end{theorem}	
						
						\begin{proof}
							Applying \cref{eq:zer:uniq} with $C$ replaced by $A+B$, we conclude that  
							$\zer(A+B)\neq \fady$ is a singleton.
							It follows from \cite[Proposition~26.1(iii)(b)]{BC2017} that $\fix T\neq \fady$. 
							\cref{thm:DR:i}:
							This is \cite[Proposition~26.11(vi)(b)]{BC2017}.
							
							\cref{thm:PR:ii:a}\&\cref{thm:PR:ii:b}:
							Suppose that $C=A^{-1}$. It follows from \cref{thm-um-self-dual} 
							that $R_A$ is a contraction for large distances. 
							Consequently, $T $ is  a contraction for large distances.
							Now combine with \cref{cor-cil-fp}\cref{cor-cil-fp:i}\&\cref{cor-cil-fp:ii}.
							
							Now suppose that $C=B^{-1}$. It follows from 
							\cref{prop:um:sne} applied to $A$
							(respectively $B^{-1}$)  that $-R_A$ and 
							(respectively $R_B=-R_{B^{-1}}$) is strongly nonexpansive.
							Consequently, $-R_BR_A$ is strongly nonexpansive by 
							\cref{prop:comp:sne}\cref{prop:comp:i} 
							applied with $(T_1,T_2)$ replaced by $(R_A,R_B)$.	
							Now combine with \cref{prop:DR:cont} applied with 
							$(\lambda,R)$ replaced by $(\tfrac{1}{2}, R_BR_A)$.
							This proves \cref{thm:PR:iii:a}. To show 
							\cref{thm:PR:iii:b},  combine 
							\cref{thm:DR:iii:a} with \cref{fact:T:sne:converges}.
						\end{proof}	
						
						We conclude this section with an application to
						the forward-backward algorithm.
						\begin{theorem}[{\bf Forward-backward algorithm}]
							\label{thm:FB}
							Let $\beta >0$.
							Suppose that $A$ is $\beta$-cocoercive
							and  that $B$ is uniformly monotone.
							Let $\gamma\in \left]0,2\beta\right[$.
							Set $T=J_{\gamma B}(\Id-\gamma A)$. Let $x_0\in X$
							and set $(\forall \nnn)$:
							\begin{equation}
								x_{n+1}=Tx_n.
							\end{equation}	
							Then the following hold:
							\begin{enumerate}
								\item 
								\label{thm:FB:i}	
								$\zer(A+B)$ is a singleton 
								\item
								\label{thm:FB:ii}	
								$T$  is a contraction for large distances.
								\item
								\label{thm:FB:iii}	
								$(x_n)_\nnn$ converges strongly
								to the unique point in $\zer(A+B)$.
							\end{enumerate} 
						\end{theorem}	
						
						\begin{proof}
							\cref{thm:FB:i}:
							Observe that $A+B$ is maximally monotone by, e.g.,
							\cite[Corollary~25.5(i)]{BC2017},
							and uniformly monotone.
							Now combine this with \cref{thm-umg}\cref{thm-umg:iii} 
							applied with $A$ replaced by $A+B$.
							
							\cref{thm:FB:ii}:	
							Observe that $\Id-\gamma A$ is $\gamma /(2\beta)$
							averaged, hence nonexpansive.
							Now combine this with
							\cref{lem-inv-res}\cref{lem-inv-res:c} 
							( applied with $A$
							replaced by $\gamma A$)
							and \cref{lem:comp:cont:large} applied with $(m,T_1,T_2) $ replaced 
							by $(2, \Id-\gamma A,J_{\gamma B})$.
							
							\cref{thm:FB:iii}:
							Combine \cref{thm:FB:ii}
							and \cref{cor-cil-fp}\cref{cor-cil-fp:ii}. 
						\end{proof}

						\small 
						
						\section*{Acknowledgements}
						
						The research of WMM was partially supported by 
						the Natural Sciences and Engineering Research Council of
						Canada Discovery Grant.

						\begin{appendices}			
							%			\crefalias{section}{appsec}
							%			\section{}
							%			\label{App:D}
							%			\begin{myproof}[Proof of \cref{ex-cont-in-large}.]
							%			\end{myproof}	
							
							\crefalias{section}{appsec}
							\section{}
							\label{App:B}
							\begin{myproof}[Finer conclusions for \cref{ex-cont-in-large-b}.]
								Let $g$ be  the inverse function of the function $x\mapsto x+\sin x$ 
								over the interval $\left]\tfrac{-\pi-2}{2},\tfrac{\pi+2}{2}\right[$
								and let $h$ be  the inverse function of the function $x\mapsto x-\sin x$ 
								over the interval $\left]\tfrac{-\pi-2}{2},\tfrac{\pi+2}{2}\right[$.
								Set 
								\begin{equation}
									A\colon \RR\to \RR\colon x\mapsto 	
									\begin{cases}
										x+1,&x\le \tfrac{-\pi-2}{4};
										\\
										g(2x)-x,&\abs{x}<\tfrac{\pi+2}{4};
										\\
										x-1, &\text{otherwise},
									\end{cases}	
								\end{equation}	
								and set
								\begin{equation}
									f\colon \RR\to \RR\colon x\mapsto 	
									\frac{1}{2}\begin{cases}
										x^2+2x,&x\le \tfrac{-\pi-2}{4};
										\\
										2xg(2x)-\tfrac{1}{2}(g(2x))^2+\cos(g(2x))-x^2-\tfrac{\pi+1}{2},&\abs{x}<\tfrac{\pi+2}{4};
										\\
										x^2-2x, &\text{otherwise}.
									\end{cases}	
								\end{equation}	
								
								Then 
								\begin{equation}
									\label{ex-cont-in-large-b:i} 	
									T=R_A,
								\end{equation}	
								\begin{equation}
									\label{ex-cont-in-large-b:ii}
									A=f' ,
								\end{equation}	
								\begin{equation}
									\label{ex-cont-in-large-b:e1} 	
									A^{-1}\colon \RR\to \RR\colon x\mapsto 	
									\begin{cases}
										x-1,&x\le \tfrac{-\pi+2}{4};
										\\
										h(2x)-x,&\abs{x}<\tfrac{\pi-2}{4};
										\\
										x+1, &\text{otherwise},
									\end{cases}	
								\end{equation}	
								and
								\begin{equation}
									\label{ex-cont-in-large-b:e2} 	
									f^*\colon \RR\to \RR\colon x\mapsto 	
									\frac{1}{2}\begin{cases}
										x^2-2x,&x\le \tfrac{-\pi+2}{4};
										\\
										2xh(2x)-\tfrac{1}{2}(h(2x))^2-\cos(g(2x))-x^2+\tfrac{\pi-1}{2},&\abs{x}<\tfrac{\pi-2}{4};
										\\
										x^2+2x, &\text{otherwise}.
									\end{cases}	
								\end{equation}	
								
								Observe that $T=R_A$ if and only if $A=((\Id+T)/2)^{-1}-\Id$.
								Now
								\begin{equation}
									\frac{1}{2}(\Id+T)\colon \RR\to \RR\colon x\mapsto 	
									\frac{1}{2}\begin{cases}
										x-1,&x\le- \tfrac{\pi}{2};
										\\
										x+\sin x,&\abs{x}<\tfrac{\pi}{2};
										\\
										x+1, &\text{otherwise}.
									\end{cases}	
								\end{equation}	
								Note that $	\tfrac{1}{2}(\Id+T)\colon \RR\to \RR$ 
								is strictly increasing and differentiable, hence continuous.
								Therefore,  $(\tfrac{1}{2}(\Id+T))^{-1}\colon \RR\to \RR$ 
								is strictly increasing and continuous.
								To this end let $(x,y)\in \RR\times \RR$. 
								Then $y=(	\tfrac{1}{2}(\Id+T))^{-1}(x)$
								if and only if $x=\tfrac{1}{2}(y+Ty)$.
								If $y\le -\tfrac{\pi}{2}$ then $x=\tfrac{1}{2}(y-1)$. Hence, $y=2x+1$ and
								$x\le \tfrac{-\pi-2}{4}$.		
								Similarly, if $y\ge \tfrac{\pi}{2}$ then $x=\tfrac{1}{2}(y+1)$. 
								Hence, $y=2x-1$ and
								$x\ge \tfrac{\pi+2}{4}$. If $y<\abs{\tfrac{\pi}{2}}$ 
								then $x=\tfrac{1}{2}(y+\sin y)$; equivalently,
								$y=g(2x)$ and $\abs{x}\le \tfrac{\pi+2}{4}$. 		
								This proves 
								\cref{ex-cont-in-large-b:i}.
								
								We now show that $f$ is an antiderivative of $A$.
								The formula for $f$ over the intervals $\left[\tfrac{\pi+2}{4},+\infty\right[$
								and $\left]-\infty, -\tfrac{\pi+2}{4}\right]$ is straightforward.
								To compute an antidrivative of $g(2x)-x$
								we use \cite{Parker55} to learn that 
								\begin{equation}
									\int (g(2x)-x)dx=\tfrac{1}{2}\Big(2xg(2x)-\tfrac{(g(x))^2}{2}+\cos(g(2x))-x^2\Big)+C.
								\end{equation}	
								The continuity of $A$ implies that  
								$g(\tfrac{\pi+2}{2})=\tfrac{\pi}{2}$ and $g(-\tfrac{\pi+2}{2})=-\tfrac{\pi}{2}$.
								This, together with the continuity of $f$, imply that $C=-\tfrac{\pi+1}{4}$.
								This proves	\cref{ex-cont-in-large-b:ii}.
								
								To prove 	\cref{ex-cont-in-large-b:e1} proceed similar to the proof of 
								\cref{ex-cont-in-large-b:i} and observe that $R_{A^{-1}}=-T$. 
								Analogously, the proof of \cref{ex-cont-in-large-b:e2}  
								is similar to that of \cref{ex-cont-in-large-b:ii}
								by observing that $A^{-1}=(f^*)'$. 								
								\begin{figure}
									\begin{center}
										\begin{tabular}{cc}
											\includegraphics[scale=0.45]{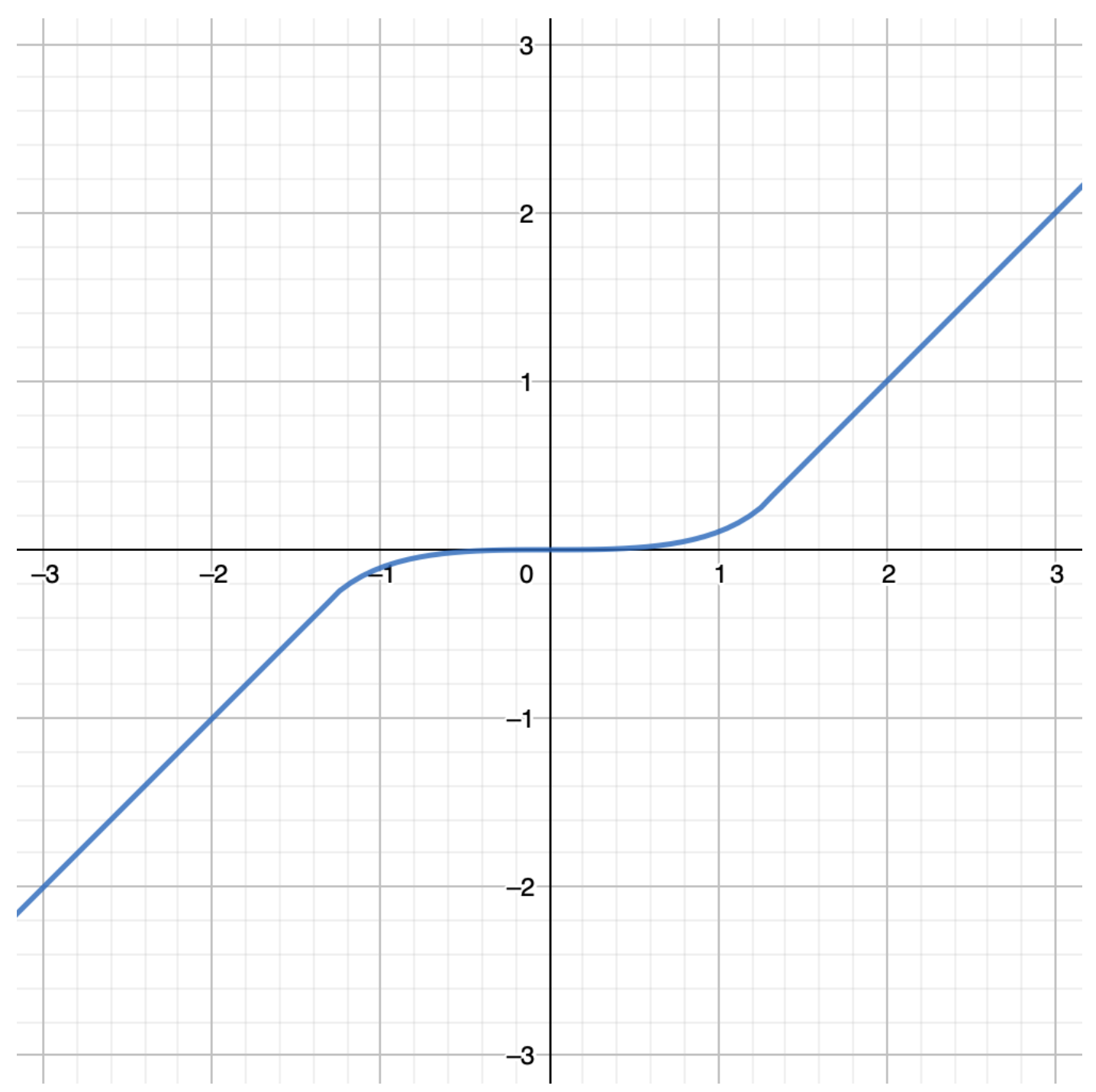}&	\includegraphics[scale=0.45]{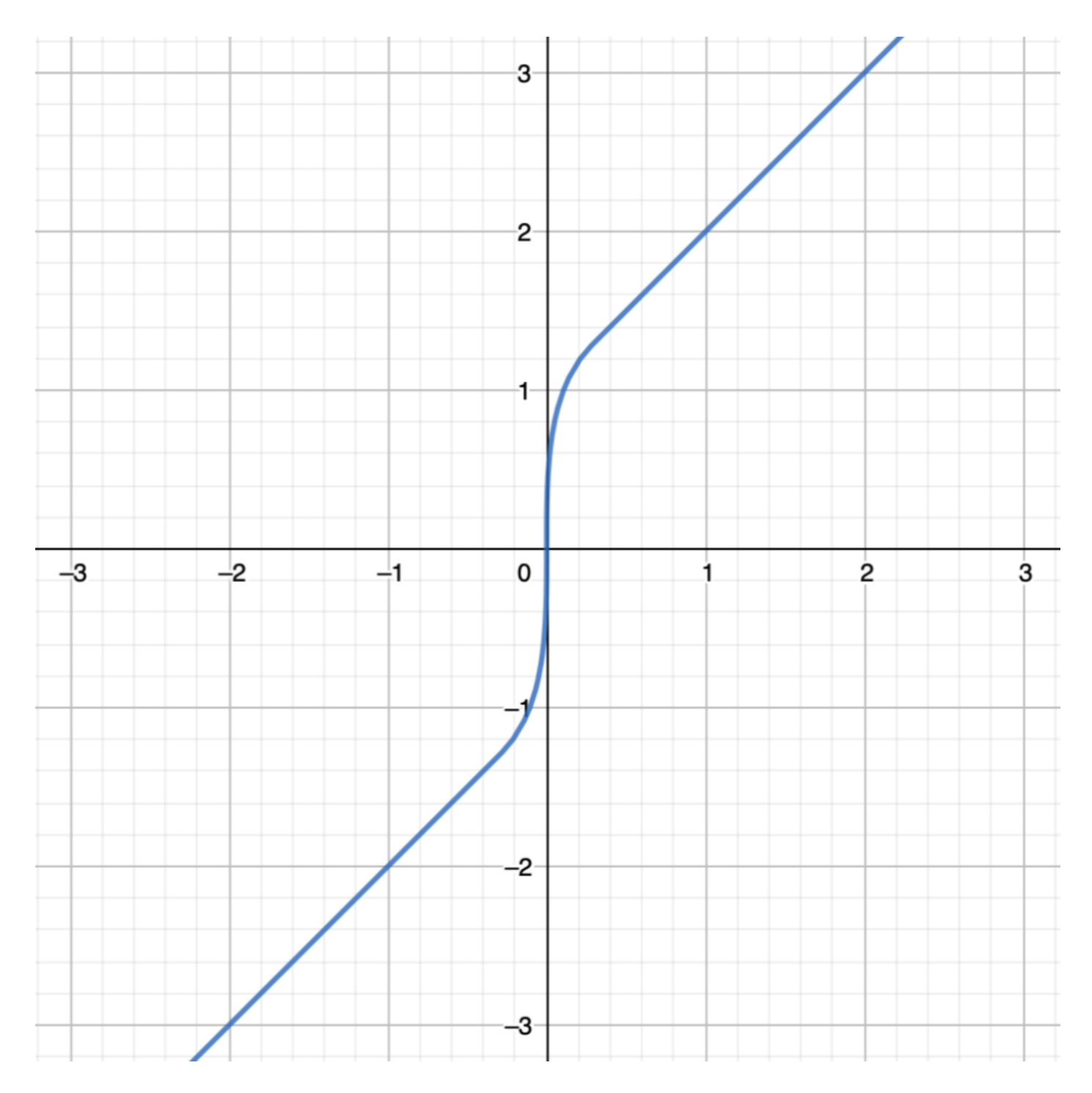}
										\end{tabular}
									\end{center}	
									\caption{A \texttt{GeoGebra} snapshot illustrating \cref{ex-cont-in-large-b}. Left: plot of $A(x)$. 
										Right: plot of $A^{-1}(x)$.}
								\end{figure}		
								\begin{figure}
									\begin{center}
										\begin{tabular}{cc}
											\includegraphics[scale=0.45]{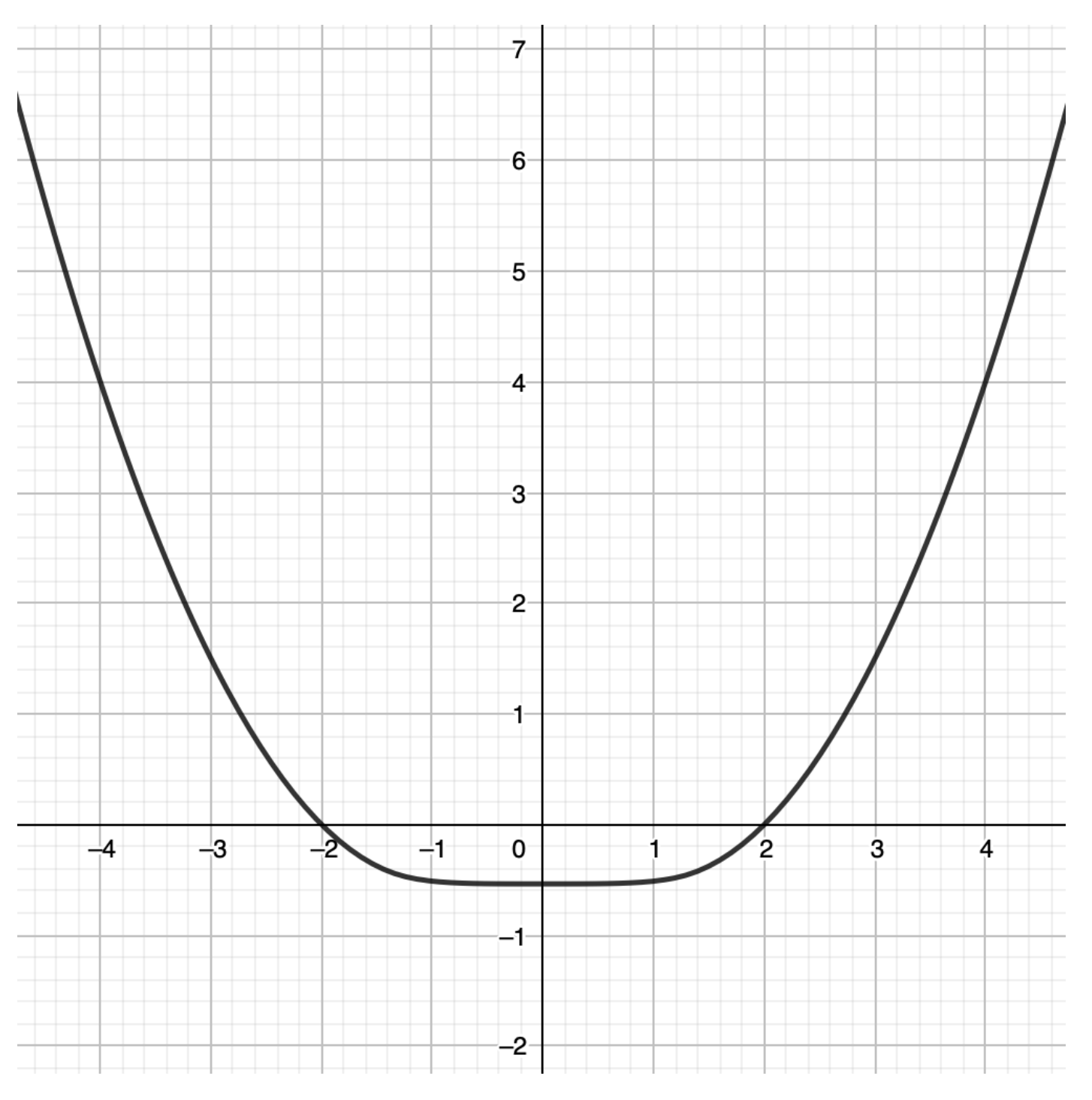}&	\includegraphics[scale=0.45]{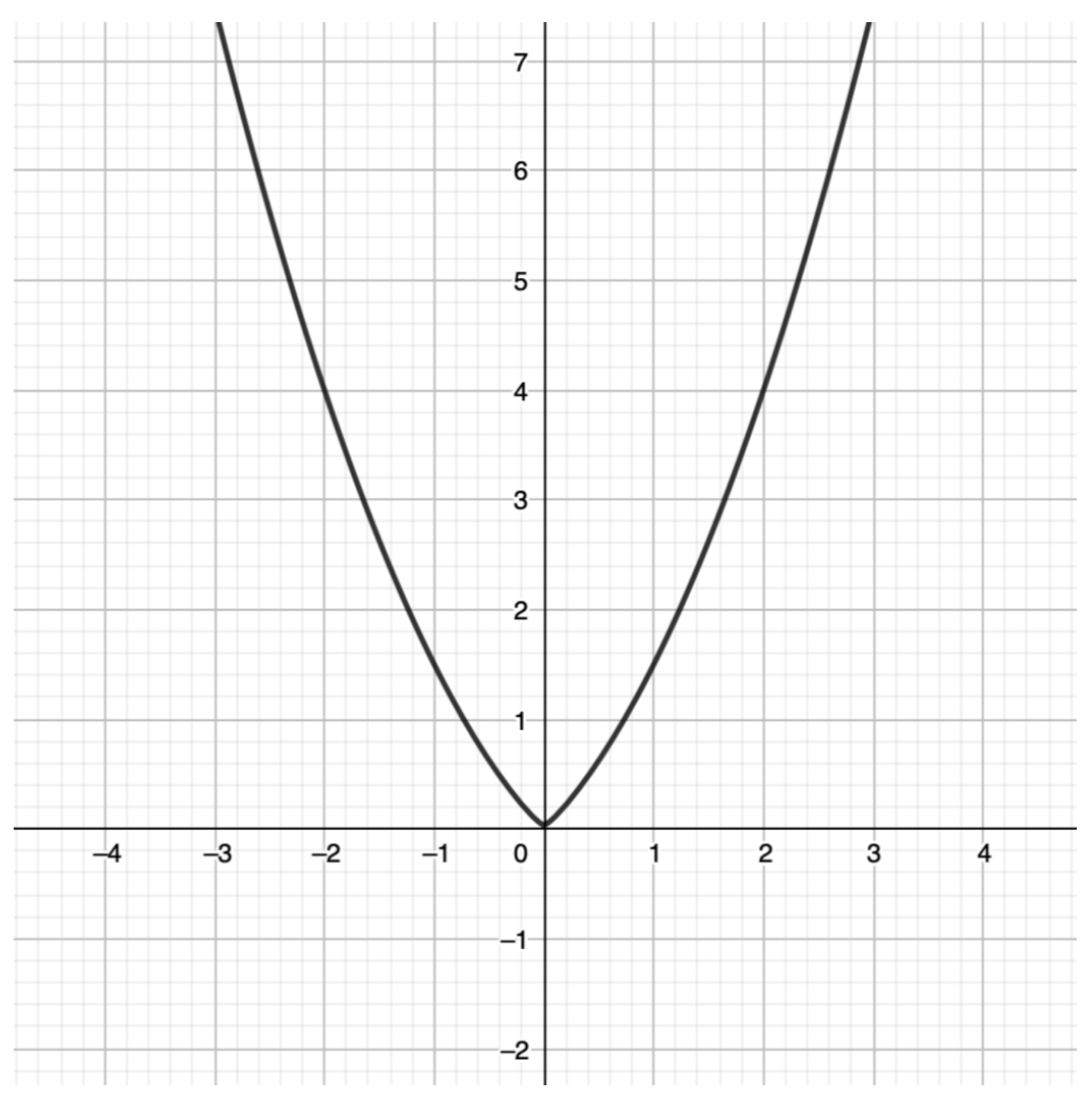}
										\end{tabular}
									\end{center}	
									\caption{A \texttt{GeoGebra} snapshot illustrating \cref{ex-cont-in-large-b}. Left: plot of $f$. 
										Right: plot of $f^*$.}
								\end{figure}
							\end{myproof}
							
							%		\crefalias{section}{appsec}
							%		\section{}
							%		\label{App:C}
							%		\begin{myproof}[Proof of \cref{ex:only:primal}.]
							%
							%			\end{myproof}
							
						\end{appendices}
						
					\end{document}